\documentclass[reqno]{amsart}
\usepackage{mathrsfs}
\usepackage{color}
\usepackage{amsmath}
\usepackage{amsfonts}
\usepackage{amssymb}
\usepackage{graphicx}
 \newtheorem{theorem}{Theorem}[section]
 \newtheorem{corollary}[theorem]{Corollary}
 \newtheorem{lemma}[theorem]{Lemma}

 \newtheorem{problem}[theorem]{Problem}
 
 \newtheorem{remark}[theorem]{Remark}
 \newtheorem{example}[theorem]{Example}
 \numberwithin{equation}{section}

\begin{document}

\title[Partial linearity and log-convexity]
 {concavity property of minimal $L^{2}$ integrals with Lebesgue measurable gain \uppercase\expandafter{\romannumeral8}: Partial linearity and log-convexity}

\author{Shijie Bao}
\address{Shijie Bao: Institute of Mathematics, Academy of Mathematics and Systems Science, Chinese Academy of Sciences, Beijing 100190, China}
\email{bsjie@amss.ac.cn}

\author{Qi'an Guan}
\address{Qi'an Guan: School of
Mathematical Sciences, Peking University, Beijing 100871, China.}
\email{guanqian@math.pku.edu.cn}

\author{Zheng Yuan}
\address{Zheng Yuan: Institute of Mathematics, Academy of Mathematics and Systems
	Science, Chinese Academy of Sciences, Beijing 100190, China.}
\email{yuanzheng@amss.ac.cn}

\thanks{}



\date{\today}

\dedicatory{}

\commby{}

%%% ----------------------------------------------------------------------

\begin{abstract}
In this article, we give some necessary conditions for the concavity property of minimal $L^2$ integrals degenerating to partial linearity, a characterization for the concavity degenerating to partial linearity for open Riemann surfaces, and some relations between the concavity property for minimal $L^2$ integrals and the log-convexity for Bergman kernels.
\end{abstract}

%%% ----------------------------------------------------------------------
\maketitle
%%% ----------------------------------------------------------------------

\section{Introduction}

Let $D\subset\mathbb{C}^n$ be a pseudoconvex domain containing the origin $o\in \mathbb{C}^n$, and let $\psi<0$ and $\varphi+\psi$ be  plurisubharmonic functions on $D$. Let $f_0$ be a holomorphic function near $o$. Let us recall  the minimal $L^2$ integrals (see \cite{guan_sharp})
$$G(t):=\inf\left\{\int_{\{\psi<-t\}}|F|^2e^{-\varphi}:F\in\mathcal{O}(\{\psi<-t\})\,\&\,(F-f_0,o)\in\mathcal{I}(\varphi+\psi)_o\right\},$$
where $t\ge0$ and $\mathcal{I}(\varphi+\psi)$ is the multiplier ideal sheaf, which was defined as the sheaf of germs of holomorphic functions $f$ such that $|f|^2e^{-\varphi-\psi}$ is locally integrable (see e.g. \cite{Tian,Nadel,Siu96,DEL,DK01,DemaillySoc,DP03,Lazarsfeld,Siu05,Siu09,DemaillyAG,Guenancia}).
 
In \cite{guan_sharp}, Guan established a concavity property for the minimal $L^{2}$ integrals $G(t)$, and a sharp version of Guan-Zhou's effectiveness result \cite{GZeff} of the strong openness property \cite{GZSOC}.

\begin{theorem}[\cite{guan_sharp}]
	\label{thm:G-concavity}If $G(0)<+\infty$, then $G(-\log r)$ is concave with respect to $r\in(0,1]$.
\end{theorem}

Applying the concavity  (Theorem \ref{thm:G-concavity}), Guan  obtained a proof of Saitoh's conjecture for conjugate Hardy $H^2$ kernels \cite{Guan2019} and a sufficient and necessary condition of the existence of decreasing equisingular approximations with analytic singularities for the multiplier ideal sheaves with weights $\log(|z_{1}|^{a_{1}}+\cdots+|z_{n}|^{a_{n}})$ \cite{guan-20}.

After that, Guan  \cite{G2018} (see also \cite{GM}) presented a concavity property on Stein manifolds with smooth gain (the weakly pseudoconvex K\"{a}hler case was obtained by Guan-Mi \cite{GM_Sci}),
which was applied by Guan-Yuan to obtain an optimal support function related to the strong openness property \cite{GY-support} and an effectiveness result of the strong openness property in $L^p$ \cite{GY-lp-effe} .

In \cite{GY-concavity}, Guan-Yuan obtained the concavity property on Stein manifolds with Lebesgue measurable gain (the weakly pseudoconvex K\"{a}hler case was obtained by Guan-Mi-Yuan \cite{GMY}, see also \cite{GMY-boundary2,GMY-boundary3}),
which was applied to deduce a twisted $L^p$ version of the strong openness property \cite{GY-twisted}.

Let us recall the concavity property of minimal $L^2$ integrals on weakly pseudoconvex K\"{a}hler manifolds with Lebesgue measurable gain (see \cite{GMY,GMY-boundary2,GMY-boundary3}). 
Let $M$ be an $n-$dimensional complex manifold. Let $X$ and $Z$ be closed subsets of $M$. A triple $(M,X,Z)$ satisfies condition $(A)$, if the following statements hold:

(1). $X$ is a closed subset of $M$ and $X$ is locally negligible with respect to $L^2$ holomorphic functions, i.e. for any coordinated neighborhood $U\subset M$ and for any $L^2$ holomorphic function $f$ on $M\setminus X$, there exists an $L^2$ holomorphic function $\tilde{f}$ on $U$ such that $\tilde{f}|_{U\setminus X}=f$ with the same $L^2$ norm;

(2). $Z$ is an analytic subset of $M$ and $M\setminus(X\cup Z)$ is a weakly pseudoconvex K\"{a}hler manifold.

\

Let $(M,X,Z)$ be a triple satisfying condition $(A)$. Let $K_M$ be the canonical line bundle on $M$. Let $\psi$ be a plurisubharmonic function on $M$, and let $\varphi$ be a Lebesgue measurable function on $M$ such that $\varphi+\psi$ is a plurisubharmonic function on $M$. Denote $T=-\sup_M\psi$.

Recall the concept of ``gain" in \cite{GMY}. A positive measurable function $c$ on $(T,+\infty)$ is in the class $\mathcal{P}_{T,M,\varphi,\psi}$ (so-called "gain") if the following two statements hold:

(1). $c(t)e^{-t}$ is decreasing with respect to $t$;

(2). there is a closed subset $E$ of $M$ such that $E\subset Z\cap\{\psi=-\infty\}$ and for any compact subset $K\subset M\setminus E$, $e^{-\varphi}c(-\psi)$ has a positive lower bound on $K$.

\

Let $Z_0$ be a subset $E$ of $M$ such that $Z_0\cap\text{Supp} (\mathcal{O}/\mathcal{I}(\varphi+\psi))\neq\emptyset$. Let $U\supset Z_0$ be an open subset of $M$, and let $f$ be a holomorphic $(n,0)$ form on $U$. Let $\mathcal{F}_{z_0}\supset\mathcal{I}(\varphi+\psi)_{z_0}$ be an ideal of $\mathcal{O}_{z_0}$ for any $z_0\in Z_0$.

Denote that
\begin{flalign}\nonumber
	\begin{split}
		G(t;\varphi,\psi,c):=\inf\bigg\{\int_{\{\psi<-t\}}|\tilde{f}|^2e^{-\varphi}c(-\psi) : \tilde{f}\in H^0(\{\psi<-t\}, \mathcal{O}(K_M))&\\
		\& (\tilde{f}-f)\in H^0(Z_0, (\mathcal{O}(K_M)\otimes\mathcal{F})|_{Z_0})\bigg\}&,
	\end{split}
\end{flalign}
and
\begin{flalign}\nonumber
	\begin{split}
		\mathcal{H}^2(c,t,\varphi,\psi):=\bigg\{\tilde{f} : \int_{\{\psi<-t\}}|\tilde{f}|^2e^{-\varphi}c(-\psi)<+\infty, \tilde{f}\in H^0(\{\psi<-t\}, \mathcal{O}(K_M))&\\
		\& (\tilde{f}-f)\in H^0(Z_0, (\mathcal{O}(K_M)\otimes\mathcal{F})|_{Z_0})\bigg\}&,
	\end{split}
\end{flalign}
where $t\in [T,+\infty)$, and $c$ is a nonnegative Lebesgue measurable function on $(T,+\infty)$. Here $|\tilde{f}|^2:=\sqrt{-1}^{n^2}\tilde{f}\wedge\bar{\tilde{f}}$ for any $(n,0)$ form $\tilde{f}$. And $(\tilde{f}-f)\in H^0(Z_0, (\mathcal{O}(K_M)\otimes\mathcal{F})|_{Z_0})$ means that $(\tilde{f}-f,z_0)\in (\mathcal{O}(K_M)\otimes\mathcal{F})_{z_0}$ for any $z_0\in Z_0$. If there is no holomorphic $(n,0)$ form $\tilde{f}$ on $\{\psi<-t\}$ satisfying $(\tilde{f}-f)\in H^0(Z_0, (\mathcal{O}(K_M)\otimes\mathcal{F})|_{Z_0})$, we set $G(t;c)=+\infty$. For simplicity, we may denote $G(t;\varphi,\psi,c)$ by $G(t)$, $G(t;\varphi,\psi)$, $G(t;\psi)$, or $G(t;c)$, denote $\mathcal{P}_{T,M,\varphi,\psi}$ by $\mathcal{P}_{T,M}$, and denote $\mathcal{H}^2(c,t,\varphi,\psi)$ by $\mathcal{H}^2(c,t)$ if there are no misunderstandings.

 The following concavity for $G(t)$ was obtained by Guan-Mi-Yuan \cite{GMY} (see also \cite{GMY-boundary2,GMY-boundary3}).

\begin{theorem}[\cite{GMY}, see also \cite{GMY-boundary2,GMY-boundary3}]
\label{Concave}
	Let $c\in\mathcal{P}_{T,M}$ satisfy $\int_{T_1}^{+\infty}c(t)e^{-t}dt<+\infty$, where $T_1>T$. If there exists $t\in [T,+\infty)$ satisfying that $G(t)<+\infty$, then $G(h^{-1}(r))$ is concave with respect to $r\in (\int_{T_1}^Tc(t)e^{-t}\mathrm{d}t,\int_{T_1}^{+\infty}c(t)e^{-t}\mathrm{d}t)$, $\lim\limits_{t\rightarrow T+0}G(t)=G(T)$ and $\lim\limits_{t\rightarrow +\infty}G(t)=0$, where $h(t)=\int_{t}^{+\infty}c(t_1)e^{-t_1}\mathrm{d}t_1$.
\end{theorem}

Guan-Mi-Yuan also obtained the following corollary of Theorem \ref{Concave}, which is a necessary condition for the concavity degenerating to linearity (some related results can be referred to \cite{GM,GY-concavity,xuzhou}).

\begin{corollary}[\cite{GMY}]\label{linear}
	Let $c\in\mathcal{P}_{T,M}$ satisfy $\int_{T_1}^{+\infty}c(t)e^{-t}dt<+\infty$, where $T_1>T$. If $G(t)\in (0,+\infty)$ for some $t\geq T$ and $G({h}^{-1}(r))$ is linear with respect to $r\in (0,\int_T^{+\infty}c(s)e^{-s}\mathrm{d}s)$, then there exists a unique holomorphic $(n,0)$ form $F$ on $M$ satisfying $(F-f)\in H^0(Z_0,(\mathcal{O}(K_M)\otimes\mathcal{F})|_{Z_0})$, and $G(t;c)=\int_{\{\psi<-t\}}|F|^2e^{-\varphi}c(-\psi)$ for any $t\geq T$.
	
	Furthermore, we have
	\begin{equation}\nonumber
		\int_{\{-t_2\leq\psi<-t_1\}}|F|^2e^{-\varphi}a(-\psi)=\frac{G(T_1;c)}{\int_{T_1}^{+\infty}c(t)e^{-t}\mathrm{d}t}\int_{t_1}^{t_2}a(t)e^{-t}\mathrm{d}t,
	\end{equation}
	for any nonnegative measurable function $a$ on $(T,+\infty)$, where $T\leq t_1<t_2\leq +\infty$.
\end{corollary}

Thus, there is a natural problem, which was posed in \cite{GY-concavity3}:

\begin{problem}[\cite{GY-concavity3}]\label{Q:chara}
How to characterize the concavity property degenerating to linearity?
\end{problem}

For  open Riemann surfaces, Guan-Yuan \cite{GY-concavity} gave an answer to Problem \ref{Q:chara} for single point, i.e. a characterization for the concavity $G(h^{-1}(r);\varphi,\psi,c)$ degenerating to linearity, where the  weights $\varphi$ may not be subharmonic and the gain $c$ may not be smooth (the case of subharmonic weights and smooth gain was proved by Guan-Mi \cite{GM}). The characterization \cite{GY-concavity} was  used by Guan-Mi-Yuan to prove a weighted version of Suita conjecture for higher derivatives \cite{GMY}, and was  used by Guan-Yuan to prove a weighted version of Saitoh's conjecture \cite{GY-saitoh}.

After that, Guan-Yuan \cite{GY-concavity3} gave an answer to Problem \ref{Q:chara} for finite points on open Riemann surfaces, which was used to obtain a characterization of the holding of equality in optimal jets $L^2$ extension problem from analytic subsets to open Riemann surfaces. Some recent results on high-dimensional manifolds can be referred to \cite{GY-concavity4,BGY-concavity5,BGY-concavity6,BGMY-concavity7}.
 
In this article, we consider the case $G(h^{-1}(r))$ being partially linear, and give some necessary conditions for the concavity property of minimal $L^2$ integrals degenerating to partial linearity. For open Riemann surfaces, we give a characterization for the concavity degenerating to partial linearity. Finally, we discuss the relations between concavity property for minimal $L^2$ integrals and log-convexity for Bergman kernels.

\subsection{$G({h}^{-1}(r))$ being partially linear}

In this section, we discuss the case that the minimal $L^2$ integral $G({h}^{-1}(r))$ is partially linear.

The following theorem gives a necessary condition for $G({h}^{-1}(r))$ is linear with respect to $r\in [\int_{T_2}^{+\infty}c(s)e^{-s}\mathrm{d}s$, $\int_{T_1}^{+\infty}c(s)e^{-s}\mathrm{d}s]$.

\begin{theorem}\label{partiallylinear}
	Let $c\in\mathcal{P}_{T,M}$ satisfy $\int_{T_1}^{+\infty}c(t)e^{-t}dt<+\infty$, where $T_1>T$. Let $T_2\in (T_1,+\infty)$. If $G(t)\in (0,+\infty)$ for some $t\geq T$ and $G({h}^{-1}(r))$ is linear with respect to $r\in [\int_{T_2}^{+\infty}c(s)e^{-s}\mathrm{d}s$, $\int_{T_1}^{+\infty}c(s)e^{-s}\mathrm{d}s]$, then there exists a unique holomorphic $(n,0)$ form $F$ on $\{\psi<-T_1\}$ satisfying $(F-f)\in H^0(Z_0,(\mathcal{O}(K_M)\otimes\mathcal{F})|_{Z_0})$, and $G(t;c)=\int_{\{\psi<-t\}}|F|^2e^{-\varphi}c(-\psi)$ for any $t\in [T_1,T_2]$.
	
	Furthermore, we have
	\begin{equation}\label{a(t)}
		\int_{\{-t_2\leq\psi<-t_1\}}|F|^2e^{-\varphi}a(-\psi)=\frac{G(T_1;c)-G(T_2;c)}{\int_{T_1}^{T_2}c(t)e^{-t}\mathrm{d}t}\int_{t_1}^{t_2}a(t)e^{-t}\mathrm{d}t,
	\end{equation}
	for any nonnegative measurable function $a$ on $[T_1,T_2]$, where $T_1\leq t_1<t_2\leq T_2$.
\end{theorem}

We give a remark on the $F$ in the above theorem.

\begin{remark}\label{tildec}
	Let $t\in [T_1,T_2)$, and let $\tilde{c}$ be a nonnegative measurable function on $[T,+\infty)$. If $\mathcal{H}^2(\tilde{c},t)\subset\mathcal{H}^2(c,t)$, and the holomorphic $(n,0)$ form $F$ in Theorem \ref{partiallylinear} satisfies
	\begin{equation}\nonumber
		\int_{\{\psi<-T_2\}}|F|^2e^{-\varphi}\tilde{c}(-\psi)=G(T_2;\tilde{c}),
	\end{equation}
	then
	\begin{equation}\nonumber
		G(t;\tilde{c})=\int_{\{\psi<-t\}}|F|^2e^{-\varphi}\tilde{c}(-\psi)=G(T_2;\tilde{c})+\frac{G(T_1;c)-G(T_2;c)}{\int_{T_1}^{T_2}c(s)e^{-s}\mathrm{d}s}\int_t^{T_2}\tilde{c}(s)e^{-s}\mathrm{d}s
	\end{equation}
	for $t\in[T_1,T_2]$.
\end{remark}

The following theorem is a sufficient condition for the minimal $L^2$ integral $G(h^{-1}(r))$ being partially linear.
\begin{theorem}\label{s-con}
Let $\psi<-T,\tilde{\psi}<-T$ be plurisubharmonic functions on $M$, and let $\varphi,\tilde{\varphi}$ be Lebesgue measurable functions on $M$ such that $\varphi+\psi,\tilde{\varphi}+\tilde{\psi}$ are plurisubharmonic functions on $M$. Let $Z_0$, $\mathcal{F}$, $\tilde{Z}_0$, $\tilde{\mathcal{F}}$ be as above with respect to $(\varphi,\psi),(\tilde{\varphi},\tilde{\psi})$.  Let $T_1,T_2\in [T,+\infty)$ with $T_1<T_2$. Let $c(t)\in\mathcal{P}_{T,M,\varphi,\psi}$ satisfying $\int_{T_2}^{+\infty}c(t)e^{-t}dt<+\infty$, $\tilde{c}(t)\in\mathcal{P}_{T,M,\tilde{\varphi},\tilde{\psi}}$ satisfying $\int_{T_2}^{+\infty}c(t)e^{-t}dt<+\infty$. Assume that

(1). $\tilde{\psi}<-T_1$ on $\{\psi<-T_1\}$, and $\tilde{\psi}<-T_2$ on $\{\psi<-T_2\}$;

(2). $\tilde{\psi}=\psi,\tilde{\varphi}=\varphi$ on $\{-T_2\leq\psi<-T_1\}$;

(3). $\tilde{Z}_0=Z_0$, $\mathcal{I}(\tilde{\psi}+\tilde{\varphi})|_{Z_0}$ $\subset\mathcal{I}(\psi+\varphi)|_{Z_0}$ and $\mathcal{F}|_{Z_0}\subset\tilde{\mathcal{F}}|_{Z_0}$;

(4). $\mathcal{H}^2(\tilde{c},t,\tilde{\varphi},\tilde{\psi})\supset\mathcal{H}^2(c,t,\varphi,\psi)$ for any $t\in [T_1,T_2]$.

If $G(\tilde{h}^{-1}(r);\tilde{\varphi},\tilde{\psi},\tilde{c})$ is linear with respect to $r\in [\int_{T_2}^{+\infty}\tilde{c}(s)e^{-s}\mathrm{d}s$, $\int_{T_1}^{+\infty}\tilde{c}(s)e^{-s}\mathrm{d}s]$, where $\tilde{h}(t)=\int_t^{+\infty}\tilde{c}(l)e^{-l}\mathrm{d}l$, and the holomorphic $(n,0)$ form $\tilde{F}$ in Theorem \ref{partiallylinear} which satisfies that $(\tilde{F}-f)\in H^0(\tilde{Z}_0,(\mathcal{O}(K_M)\otimes\tilde{\mathcal{F}})|_{\tilde{Z}_0})$ and
\begin{equation}\label{s-con1}
	\int_{\{\tilde{\psi}<-t\}}|\tilde{F}|^2e^{-\tilde{\varphi}}\tilde{c}(-\tilde{\psi})=G(t;\tilde{\varphi},\tilde{\psi},\tilde{c}), \ \forall t\in [T_1,T_2],
\end{equation}
also satisfies that
	\begin{equation}\nonumber
	\int_{\{\psi<-T_2\}}|F|^2e^{-\varphi}c(-\psi)=G(T_2;\varphi,\psi,c),
\end{equation}
Then $G(h^{-1}(r);\varphi,\psi,c)$ is linear with respect to $r\in [\int_{T_2}^{+\infty}c(s)e^{-s}\mathrm{d}s$, $\int_{T_1}^{+\infty}c(s)e^{-s}\mathrm{d}s]$, where $h(t)=\int_t^{+\infty}c(l)e^{-l}\mathrm{d}l$.
\end{theorem}

Applying Theorem \ref{partiallylinear}, we obtain the following necessary condition for $\varphi+\psi$ when $G(h^{-1}(r))$ is partially linear.

\begin{theorem}\label{sp1}
	Let $(M,X,Z)$ be a triple satisfying condition $(A)$, and let $\psi$ be a plurisubharmonic function on $M$.  Let $\varphi$ be a Lebesgue measurable function on $M$ such that $\varphi+\psi$ is a plurisubharmonic function on $M$. Let $T_1,T_2\in [T,+\infty)$ with $T_1<T_2$. Let $c\in\mathcal{P}_{T,M}$ satisfy $\int_{T_2}^{+\infty}c(t)e^{-t}dt<+\infty$.  Assume that there exists $t\geq T$ such that $G(t)\in (0,+\infty)$.  If there exists a Lebesgue measurable function $\tilde{\varphi}$ on $M$ such that:
	
	(1). $\tilde{\varphi}+\psi$ is plurisubharmonic function on $M$;
	
	(2). $\tilde{\varphi}=\varphi$ on $\{\psi<-T_2\}$;
	
	(3). $\tilde{\varphi}\geq \varphi$ and $\tilde{\varphi}\not\equiv\varphi$ on the interior  of $\{-T_2\leq\psi<-T_1\}$ (assume that the interior is not empty);
	
	(4). $\lim\limits_{t\rightarrow T_1+0}\sup\limits_{\{-t\leq \psi<-T_1\}}(\tilde{\varphi}-\varphi)=0$;
	
	(5). $\tilde{\varphi}-\varphi$ is bounded on $\{\psi<-T_1\}$, \\
	then $G(h^{-1}(r))$ is not linear with respect to $r\in [\int_{T_2}^{+\infty}c(s)e^{-s}\mathrm{d}s,\int_{T_1}^{+\infty}c(s)e^{-s}\mathrm{d}s]$.
	
	Especially, if $\varphi+\psi$ is strictly plurisubharmonic at some $z_1$, where $z_1$ is a point in the interior of $\{-T_2\leq\psi<-T_1\}$, then $G(h^{-1}(r))$ is not linear with respect to $r\in [\int_{T_2}^{+\infty}c(s)e^{-s}\mathrm{d}s,\int_{T_1}^{+\infty}c(s)e^{-s}\mathrm{d}s]$.
\end{theorem}

Applying Theorem \ref{partiallylinear}, we obtain the following necessary condition for $\psi$ when $G(h^{-1}(r))$ is partially linear.

\begin{theorem}\label{sp2}
	Let $(M,X,Z)$ be a triple satisfying condition $(A)$, and let $\psi$ be a plurisubharmonic function on $M$. Let $\varphi$ be a Lebesgue measurable function on $M$ such that $\varphi+\psi$ is a plurisubharmonic function on $M$. Assume that there exists $t\geq T$ such that $G(t)\in (0,+\infty)$. Let $T_1,T_2\in [T,+\infty)$ with $T_1<T_2$. Let $c\in\mathcal{P}_{T,M}$ satisfy $\int_{T_2}^{+\infty}c(t)e^{-t}dt<+\infty$. If there exists a plurisubharmonic function $\tilde{\psi}$ on $M$ such that:
	
	(1). $\tilde{\psi}<-T_1$ on $\{\psi<-T_1\}$;
	
	(2). $\tilde{\psi}=\psi$ on $\{\psi<-T_2\}$;
	
	(3). $\tilde{\psi}\geq \psi$ and $\tilde{\psi}\not\equiv\psi$ on the interior  of $\{-T_2\leq\psi<-T_1\}$ (assume that the interior is not empty);
	
	(4). $\lim\limits_{t\rightarrow T_2-0}\sup\limits_{\{-T_2\leq \psi<-t\}}(\tilde{\psi}-\psi)=0$, \\
	then $G(h^{-1}(r))$ is not linear with respect to $r\in [\int_{T_2}^{+\infty}c(s)e^{-s}\mathrm{d}s,\int_{T_1}^{+\infty}c(s)e^{-s}\mathrm{d}s]$.
	
	Especially, if $\psi$ is strictly plurisubharmonic at some $z_1$, where $z_1$ is a point in the interior of $\{-T_2\leq\psi<-T_1\}$, then $G(h^{-1}(r))$ is not linear with respect to $r\in [\int_{T_2}^{+\infty}c(s)e^{-s}\mathrm{d}s,\int_{T_1}^{+\infty}c(s)e^{-s}\mathrm{d}s]$.
\end{theorem}

We give an example of the partially linear case.
\begin{example}
	Let $M=\Delta$ be the unit disc in $\mathbb{C}$, and $Z_0=o$ be the origin. Let $g<0$ be a increasing convex function on $(-\infty,0)$. Take $\psi=g(\log|z|)$ and $\varphi=2\log|z|-g(\log|z|)$ on $\Delta$. Let $f\equiv dz$, $c\equiv1$ and $\mathcal{F}_{o}=(z)_o$. It is clear that 
	$$G(t)=\int_{\{\psi<-t\}}|dz|^2e^{-\varphi}=4\pi\int_{0}^{e^{g^{-1}(-t)}}se^{-2\log s+g(\log s)}ds$$
	for any $t\ge0$.
	Thus, for any $a,b\in[0,1]$ $(a<b)$,  $G(-\log r)$ is linear on $[a,b]$ if and only if $g'=const$ on $(\log a,\log b)$.
\end{example}

\subsection{A characterization for  the concavity degenerating to partial linearity}\label{sec-characterization}
In this section, we give a characterization for  the concavity degenerating to partial linearity on open Riemann surfaces.

Let $\Omega$ be an open Riemann surface, which admits a nontrivial Green function $G_{\Omega}$. 
Let $a$ be a negative number, and let $z_0\in \Omega$. Let $$\psi=G_{\Omega}(\cdot,z_0)+\max\{G_{\Omega}(\cdot,z_0),a\}$$
 on $\Omega$, and let 
 $$\varphi=\varphi_0+\min\{G_{\Omega}(\cdot,z_0),a\}$$
  on $\Omega$, where $\varphi_0$ is a subharmonic function on $\Omega$.  Let $c(t)$ be a Lebesgue measurable function on $[0,+\infty)$, such that $c(t)e^{-t}$ is decreasing on $[0,+\infty)$ and $\int_0^{+\infty}c(t)e^{-t}dt<+\infty$.

Let $w$ be a local coordinate on a neighborhood $V_{z_0}$ of $z_0$ satisfying $w(z_0)=0$. Let $f_0$ be a holomorphic $(1,0)$ form on $V_{z_0}$. Denote that 
\begin{displaymath}
	\begin{split}
		G(t):=\inf\Bigg\{\int_{\{\psi<-t\}}|\tilde F|^2e^{-\varphi}c(-\psi):&\tilde F\in H^0(\{\psi<-t\},\mathcal{O}(K_{\Omega}))\\
		&\&\,(\tilde F-f_0,z_0)\in(\mathcal{O}(K_{\Omega})\otimes\mathcal{I}(\varphi+\psi))_{z_0}\Bigg\}
	\end{split}
\end{displaymath}
Note that $\psi$ and $\varphi+\psi$ are subharmonic on $\Omega$. Then $G(h^{-1}(r))$ is concave on $[0,\int_0^{+\infty}c(t)e^{-t}dt]$ by Theorem \ref{Concave}, where $h(t)=\int_t^{+\infty}c(s)e^{-s}ds$ for $t\ge0$.

Let us recall some notations (see \cite{OF81,GY-concavity,GMY}).
Let $p:\Delta\rightarrow \Omega$ be the universal covering from unit disc $\Delta$ to $\Omega$.
For any function $u$ on $\Omega$ with value $[-\infty,+\infty)$ such that $e^u$ is locally the modulus of a holomorphic function, there exist a character $\chi_{u}$ (a representation of the fundamental group of $\Omega$) and a holomorphic function $f_u$ on $\Delta$,
such that $|f_u|=p^{*}\left(e^{u}\right)$ and  $g^{*}f=\chi(g)f$ for any element $g$ of the fundamental group of $\Omega$,
where $|\chi|=1$. If $u_1-u_2=\log|f|$, where $f$ is a holomorphic function on $\Omega$, then $\chi_{u_1}=\chi_{u_2}$.
For the Green function $G_{\Omega}(\cdot,z_0)$, denote that $\chi_{z_0}:=\chi_{G_{\Omega}(\cdot,z_0)}$ and $f_{z_0}:=f_{G_{\Omega}(\cdot,z_0)}$.

We give a characterization for  the concavity of $G(h^{-1}(r))$ degenerating to partial linearity.
\begin{theorem}
	\label{thm:char}
	Assume that $\{z\in\Omega:G_{\Omega}(z,z_0)=a\}\Subset\Omega$ and $G(0)\in(0,+\infty)$. Then $G(h^{-1}(r))$ is linear on $[0,\int_{-2a}^{+\infty}c(t)e^{-t}dt]$ and $[\int_{-2a}^{+\infty}c(t)e^{-t}dt,\int_0^{+\infty}c(t)e^{-t}dt]$ if and only if the following statements hold:
	
	$(1)$ $\varphi_0=2\log|g|+2u$ on $\Omega$, where $u$ is a harmonic function on $\Omega$ and $g$ is a holomorphic function on $\Omega$ satisfying that $ord_{z_0}(g)=ord_{z_0}(f_0)$;
	
	$(2)$ $\chi_{-u}=\chi_{z_0}$, where $\chi_{-u}$ and $\chi_{z_0}$ is the characters associated to the functions $-u$ and $G_{\Omega}(\cdot,z_0)$ respectively.
\end{theorem}

\begin{remark}\label{r:not linear}
 By Theorem 1.17 in \cite{GY-concavity} (see also Theorem \ref{thm:e2}), $G(h^{-1}(r))$ is not linear on $[0,\int_0^{+\infty}c(t)e^{-t}dt]$.
\end{remark}

\subsection{Relations between the concavity property for minimal $L^2$ integrals and the log-convexity for Bergman kernels}

\

Recall the definition of the weighted Bergman kernel on a bounded planar domain $D\subset \mathbb{C}$. For any $z\in D$, the weighted Bergman kernel on $D$ is defined by
\[B_{D,\varphi}(z):=\left(\inf\left\{\int_D|f|^2e^{-\varphi} : f\in\mathcal{O}(D), \ f(z)=1\right\}\right)^{-1},\]
where $\varphi$ is a subharmonic function on $D$.
Let $\phi$ be a negative subharmonic function on $D$ such that $\phi(z_0)=-\infty$, where $z_0\in D$. Berndtsson's subharmonicity of the fiberwise Bergman kernels (\cite{Blogsub}) implies the log-convexity of the Bergman kernels on the sublevel sets of $\phi$, i.e., the function $\log B_{\{\phi<-t\}\cap D,\varphi}(z_0)$ is convex with respect to $t\in [0,+\infty)$ (see \cite{BL16}).  

Note that for $\phi=2G_D(\cdot,z_0)$, the reciprocal of the weighted Bergman kernel on the sublevel set equals to some specific minimal $L^2$ integral mentioned above. Actually, we have
\[\big(B_{\{\phi<-t\}\cap D,\varphi}(z_0)\big)^{-1}=\mathscr{G}(t),\]
where
\begin{align*}
\begin{split}
    \mathscr{G}(t):=\inf\bigg\{\int_{\{2G(\cdot,z_0)<-t\}}|f|^2e^{-\varphi} : &f\in \mathcal{O}(\{2G(\cdot,z_0)<-t\})\\
    &\& (f-1,z_0)\in\mathcal{I}(2G(\cdot,z_0))_{z_0}\bigg\}.
    \end{split}
\end{align*}
It follows that $-\log\mathscr{G}(t)$ is convex with respect to $t\in [0,+\infty)$. Besides, Theorem \ref{Concave} gives that $\mathscr{G}(-\log r)$ is concave with respect to $r\in (0,1]$. We can find that, if we note that $\log \mathscr{G}(t)+t$ is lower bounded on $(0,+\infty)$ (see \cite{BL16}), the log-convexity of the Bergman kernels can imply the concavity of $\mathscr{G}(-\log r)$,  which follows from a simple calculation (see Lemma \ref{l:concave}).

It is natural to ask:
\begin{problem}\label{q:2}
Is $-\log G(t)$ convex on $[0,+\infty)$ for general minimal $L^2$ integral $G(t)$?
\end{problem}

Note that, in the definition of $\mathscr{G}(t)$, the weight function $\varphi$ is subharmonic and the ideal $\mathcal{I}(2G(\cdot,z_0))_{z_0}$ is the maximal ideal of $\mathcal{O}_{z_0}$. If one of the two conditions does not hold, we can give negative answers to Problem \ref{q:2}, i.e., $-\log G(t)$ is not convex.

 Follow the notations and assumptions in Theorem \ref{thm:char} in the following theorem.

\begin{theorem}
	\label{counterexample2}Assume that statements $(1)$ and $(2)$ in Theorem \ref{thm:char} hold. Then $-\log G(t)$ is not convex on $[0,+\infty)$.
\end{theorem}

Theorem \ref{counterexample2} shows that $-\log G(t)$ may not be convex if the weight $\varphi$ is not subharmonic. 
In the following, we consider the case that the weight is trivial and the ideal $\mathcal{F}_{z_0}$ is not the maximal ideal of $\mathcal{O}_{z_0}$. 

 Let $\Omega\subset\mathbb{C}$ be a domain bounded by finite analytic curves, and let $z_0\in\Omega$. Denote that 
\begin{displaymath}
	\begin{split}
		G_k(t):=\inf\Bigg\{\int_{\{2(k+1)G_{\Omega}(\cdot,z_0)<-t\}}&|F|^2:F\in\mathcal{O}(\{2(k+1)G_{\Omega}(\cdot,z_0)<-t\}),\\
		&(F-1,z_0)\in\mathcal{I}(2(k+1)G_{\Omega}(\cdot,z_0))_{z_0}\Bigg\},	
	\end{split}
\end{displaymath}
where $t\ge0$.
$G_k(-\log r)$ is concave on $(0,1]$ by Theorem \ref{thm:G-concavity} and   $-\log G_0(t)$ is convex on $(0,+\infty)$ (see \cite{BL16}). 

The following theorem gives a negative answer to Problem \ref{q:2}.
\begin{theorem}
	\label{counterexample}If $\Omega$ is not a
	disc ($\Omega$ may be conformally equivalent to a disc), then there exists large enough $k>0$ such that $-\log G_k(t)$ is not convex on $[0,+\infty)$.
\end{theorem}

We give an example on the disc as follows:
\begin{example}\label{e:counterexample3}
	Let $k\ge1$ be an integer. Let $\psi:=2(k+1)\log|z|$ on $\Delta$, and let $a_j\in\mathbb{C}$ for any $0\le j\le k$. Denote that $G(t):=\inf\{\int_{\{\psi<-t\}}|F|^2:\mathcal{O}(\{\psi<-t\})\,\&\,f^{(j)}(o)=j!a_j$ for any $0\le j\le k\}$ for $t\ge0$. Note that $G(-\log r)$ is concave on $(0,1]$ by Theorem \ref{thm:G-concavity}. If there exist $j_1$ and $j_2$ satisfying $j_1\not=j_2$, $a_{j_1}\not=0$ and $a_{j_2}\not=0$, then $-\log G(t)$ is not convex on $[0,+\infty)$ (the proof is given in Section \ref{sec:p4}).
\end{example}

\section{Preparations}
In this section, we do some preparations.

\subsection{Some lemmas about minimal $L^2$ integrals}

We recall some lemmas, which will be used in the discussion of minimal $L^2$ integrals.

\begin{lemma}[see \cite{GMY}]\label{dbarequa}
	Let $(M,X,Z)$ be a triple satisfying condition $(A)$ and $c(t)\in\mathcal{P}_{T,M}$. Let $B\in (0,+\infty)$ and $t_0>t_1>T$ be arbitrary given. Let $\psi<-T$ be a plurisubharmonic function on $M$. Let $\varphi$ be a plurisubharmonic function on $M$. Let $F$ be a holomorphic $(n,0)$ form on $\{\psi<-t_0\}$ such that
	\begin{equation}\nonumber
		\int_{\{\psi<-t_0\}}|F|^2e^{-\varphi}c(-\psi)<+\infty.
	\end{equation}
Then there exists a holomorphic $(n,0)$ form $\tilde{F}$ on $\{\psi<-t_1\}$ such that
\begin{equation}\nonumber
	\int_{\{\psi<-t_1\}}|\tilde{F}-(1-b_{t_0,B}(\psi))F|^2e^{-\varphi-\psi+v_{t_0,B}(\psi)}c(-v_{t_0,B}(\psi))\leq C\int_{t_1}^{t_0+B}c(t)e^{-t}\mathrm{d}t,
\end{equation}
where
\begin{equation}\nonumber
	C=\int_M\frac{1}{B}\mathbb{I}_{\{-t_0-B<\psi<-t_0\}}|F|^2e^{-\varphi-\psi},
\end{equation}
$b_{t_0,B}(t)=\int_{-\infty}^t\frac{1}{B}\mathbb{I}_{\{-t_0-B<s<-t_0\}}\mathrm{d}s$, and $v_{t_0,B}(t)=\int_{-t_0}^tb_{t_0,B}\mathrm{d}s-t_0$.
\end{lemma}

Following the assumptions and notations in Theorem \ref{Concave}, we recall three properties about $G(t)$.
\begin{lemma}[see \cite{GMY}]\label{G(t)=0}
	The following three statements are equivalent:
	
	(1). $f\in H^0(Z_0,(\mathcal{O}(K_M)\otimes\mathcal{F})|_{Z_0})$;
	
	(2). $G(t)=0$ for some $t\geq T$;
	
	(3). $G(t)=0$ for any $t\geq T$.
\end{lemma}

The following lemma shows the existence of minimal form. 
\begin{lemma}[see \cite{GMY}]\label{F_t}
	Assume that $G(t)<+\infty$ for some $t\in [T,+\infty)$. Then there exists a unique holomorphic $(n,0)$ form $F_t$ on $\{\psi<-t\}$ satisfying
	\begin{equation}\nonumber
		\int_{\{\psi<-t\}}|F_t|^2e^{-\varphi}c(-\psi)=G(t)
	\end{equation}
and $(F_t-f)\in H^0(Z_0,(\mathcal{O}(K_M)\otimes\mathcal{F})|_{Z_0})$.

Furthermore, for any holomorphic $(n,0)$ form $\hat{F}$ on $\{\psi<-t\}$ satisfying
\begin{equation}\nonumber
	\int_{\{\psi<-t\}}|\hat{F}|^2e^{-\varphi}c(-\psi)<+\infty
\end{equation}
and $(\hat{F}-f)\in H^0(Z_0,(\mathcal{O}(K_M)\otimes\mathcal{F})|_{Z_0})$, we have the following equality,
\begin{flalign}\nonumber
\begin{split}
	&\int_{\{\psi<-t\}}|F_t|^2e^{-\varphi}c(-\psi)+\int_{\{\psi<-t\}}|\hat{F}-F_t|^2e^{-\varphi}c(-\psi)\\
	=&\int_{\{\psi<-t\}}|\hat{F}|^2e^{-\varphi}c(-\psi).
\end{split}	
\end{flalign}
\end{lemma}

\begin{lemma}[see \cite{GMY}]\label{l-cont}
	$G(t)$ is decreasing with respect to $t\in [T,+\infty)$. $\lim\limits_{t\rightarrow t_0+0}G(t)=G(t_0)$ for any $t_0\in [T,+\infty)$. And if $G(t)<+\infty$ for some $t>T$, then $\lim\limits_{t\rightarrow +\infty}G(t)=0$. Especially, $G(t)$ is lower semicontinuous on $[T,+\infty)$.
\end{lemma}

We also recall the following two lemmas.
\begin{lemma}[see \cite{G-R}]\label{module}
	Let $N$ be a submodule of $\mathcal{O}_{\mathbb{C}^n,o}^q$, $q\in\mathbb{Z}_+\cup\{\infty\}$. Let $\{f_j\}\subset\mathcal{O}_{\mathbb{C}^n}(U)^q$ be a sequence of $q-$tuples holomorphic functions in an open neighborhood $U$ of the origin $o$. Assume that $\{f_j\}$ converges uniformly in $U$ towards a $q-$tuples $f\in\mathcal{O}_{\mathbb{C}^n,o}^q$, and assume furthermore that all the germs $(f_j,o)$ belong to $N$. Then $(f,o)\in N$.
\end{lemma}

\begin{lemma}[see \cite{GY-concavity}]\label{s_K}
	Let $M$ be a complex manifold. Let $S$ be an analytic subset of $M$. Let $\{g_j\}_{j=1,2,\ldots}$ be a sequence of nonnegative Lebesgue measurable functions on $M$, which satisfies that $g_j$ are almost everywhere convergent to $g$ on $M$ when $j\rightarrow +\infty$, where $g$ is a nonnegative Lebesgue measurable function on $M$. Assume that for any compact subset $K$ of $M\setminus S$, there exist $s_K\in (0,+\infty)$ and $C_K\in (0,+\infty)$ such that
	\begin{equation}\nonumber
		\int_Kg_j^{-s_K}\mathrm{d}V_M\leq C_K
	\end{equation}
	for any $j$, where $\mathrm{d}V_M$ is a continuous volume form on $M$.
	
	Let $\{F_j\}_{j=1,2,\ldots}$ be a sequence of holomorphic $(n,0)$ form on $M$. Assume that $\liminf\limits_{j\rightarrow +\infty}\int_M|F_j|^2g_j\leq C$, where $C$ is a positive constant. Then there exists a subsequence $\{F_{j_l}\}_{l=1,2,\ldots}$, which satisfies that $\{F_{j_l}\}$ is uniformly convergent to a holomorphic $(n,0)$ form $F$ on $M$ on any compact subset of $M$ when $l\rightarrow +\infty$, such that
	\begin{equation}\nonumber
		\int_M|F|^2g\leq C.
	\end{equation}
\end{lemma}

\subsection{Some results on open Riemann surfaces}
	 Let $\Omega$ be an open Riemann surface, which admits a nontrivial Green function $G_{\Omega}$. Let $z_0\in\Omega$. Let $\psi$ be a subharmonic function on $\Omega$ such that $(\psi-pG_{\Omega}(\cdot,z_0))(z_0)>-\infty$, where $p>0$ is a constant, and let $\varphi$ be a Lebesgue measurable function on $\Omega$ such that $\varphi+\psi$ is subharmonic on $\Omega$.
	 
	Let $f_0$ be a holomorphic $(1,0)$ form on a neighborhood of $z_0$. Let $c(t)$ be a function on $[0,+\infty)$, which satisfies that $c(t)e^{-t}$ is decreasing, $\int_0^{+\infty}c(t)e^{-t}dt<+\infty$ and $e^{-\varphi}c(-\psi)$ has a positive lower bound on any compact subset of $\Omega\backslash\{z_0\}$.
	Denote that 
\begin{displaymath}
	\begin{split}
		G(t):=\inf\Bigg\{\int_{\{\psi<-t\}}|f|^2e^{-\varphi}c(-\psi):&f\in H^0(\{\psi<-t\},\mathcal{O}(K_{\Omega}))\\
	&\&\,(f-f_0,z_0)\in(\mathcal{O}(K_{\Omega})\otimes\mathcal{I}(\varphi+\psi))_{z_0}	\Bigg\}
	\end{split}
\end{displaymath}
for any $t\ge0$. 

By Theorem \ref{Concave}, $G(h^{-1}(r))$ is concave with respect to $r\in[0,\int_0^{+\infty}c(t)e^{-t}dt]$, where	
	 ${h}(t)=\int_{t}^{+\infty}c(s)e^{-s}ds$ for any $t\ge0$. The following theorem give a characterization for  $G(h^{-1}(r))$ degenerating to linearity. 	
	
\begin{theorem}[\cite{GY-concavity}]
	\label{thm:e2}
  Assume that $G(0)\in(0,+\infty)$. Then $G(h^{-1}(r))$ is linear with respect to $r$ on $[0,\int_0^{+\infty}c(t)e^{-t}dt]$ if and only if the following statements hold:
	
	(1) $\varphi+\psi=2\log|g|+2G_{\Omega}(z,z_0)+2u$ and $ord_{z_0}(g)=ord_{z_0}(f_0)$, where $g$ is a holomorphic function on $\Omega$ and $u$ is a harmonic function on $\Omega$;
	
	(2) $\psi=2pG_{\Omega}(z,z_0)$ on $\Omega$ for some $p>0$;
	
	(3) $\chi_{-u}=\chi_{z_0}.$
\end{theorem}

\begin{remark}
	\label{r:e2}
	When the three statements in Theorem \ref{thm:e2} hold,
	$$b_0gp_*(f_udf_{z_0})$$
	is the unique holomorphic $(1,0)$ form $F$ on $\Omega$ such that $(F-f,z_0)\in(\mathcal{O}(K_{\Omega}))_{z_0}\otimes\mathcal{I}(\varphi+\psi)_{z_0}$  and
	$G(t)=\int_{\{\psi<-t\}}|F|^2e^{-\varphi}c(-\psi)$ for any $t\ge0$, where $p:\Delta\rightarrow\Omega$ is the universal covering, $f_u$ and are holomorphic functions on $\Delta$ satisfying $|f_u|=p^*(e^{u})$ and $|f_{z_0}|=p^*\left(e^{G_{\Omega}(\cdot,z_0)}\right)$, and $b_0$ is a constant such that $ord_{z_0}(b_0gp_*(f_udf_{z_0})-f_0)>ord_{z_0}(f_0)$. 
\end{remark}

 Let $w$ be a local coordinate on a neighborhood $V_{z_0}$ of $z_0\in\Omega$ satisfying $w(z_0)=0$. Let us recall some properties about Green functions.

 \begin{lemma}[see \cite{S-O69}, see also \cite{Tsuji}]
 	\label{l:green}
 	$G_{\Omega}(z,z_0)=\sup_{v\in\Delta_0(z_0)}v(z)$, where $\Delta_0(z_0)$ is the set of negative subharmonic functions $v$ on $\Omega$ satisfying that $v-\log|w|$ has a locally finite upper bound near $z_0$.
 \end{lemma}

\begin{lemma}[see \cite{GY-concavity}]\label{l:G-compact}
	For any open neighborhood $U$ of $z_0$, there exists $t>0$ such that $\{G_{\Omega}(z,z_0)<-t\}$ is a relatively compact subset of $U$.
\end{lemma}

The following lemma will be used in the proof of Lemma \ref{l:extra}.

\begin{lemma}
	\label{l:connected}
	For any $a<0$, $\{z\in\Omega: G_{\Omega}(z,z_0)<a\}$ is a connected set. 
\end{lemma}

\begin{proof}
	We prove Lemma \ref{l:connected} by contradiction: if not, there a connected component $V_1\not=\emptyset$ of $\{z\in\Omega: G_{\Omega}(z,z_0)<a\}$ such that $z_0\not\in V_1$. Denote that $\{z\in\Omega: G_{\Omega}(z,z_0)<a\}\backslash V_1=V_2$. Then there exists an open subset $U_1$ of $\Omega$, which satisfies that 
	$$V_1\subset\overline{V}_1\subset U_1\,\,\,\,\text{and}\,\,\,\,V_2\in\Omega\backslash U_1.$$
	Set \begin{equation}\nonumber
				v(z)=\left\{
		\begin{array}{ll}
			\max\{G_{\Omega}(z,z_0),a\}, &  z\in U_1\\
			G_{\Omega}(z,z_0), & z\in \Omega\backslash U_1
		\end{array}
		\right. 
	\end{equation}	
on $\Omega$. Note that $v=G_{\Omega}(\cdot,z_0)$ is subharmonic on $\Omega\backslash\overline{V}_1$ and $v=\max\{G_{\Omega}(\cdot,z_0),a\}$	is subharmonic on $U_1$, then we know that $v$ is subharmonic on $\Omega$. Note that 
$$0>v\ge G_{\Omega}(\cdot,z_0)$$
 on $\Omega$. By Lemma \ref{l:green}, we have $v=G_{\Omega}(\cdot,z_0)$,  which contradicts to $G_{\Omega}(\cdot,z_0)<a$ on $V_1$.
Thus, Lemma \ref{l:connected} holds.		
\end{proof}

The following lemma will be used in the proof of Theorem \ref{thm:char}.

\begin{lemma}\label{l:extra}
	Let $a_0<0$, and assume that $\{z\in\Omega:G_{\Omega}(z,z_0)=a_0\}\Subset\Omega.$ Then for any open neighborhood $U$ of $\{z\in\Omega: G_{\Omega}(z,z_0)=a_0\}$, there exists $a_1>a_0$ such that $\{z\in\Omega:a_0<G_{\Omega}(\cdot,z_0)<a_1\}\Subset U$.
\end{lemma}
\begin{proof}
	As $\{z\in\Omega: G_{\Omega}(z,z_0)<a_0\}$ is a connected set by Lemma \ref{l:connected} and $\{z\in\Omega:G_{\Omega}(z,z_0)=a_0\}\Subset\Omega$, without loss of generality, assume that $\tilde U:=U\cup\{z\in\Omega: G_{\Omega}(z,z_0)<a_0\}$ is connected and $U\Subset\Omega$. 
	
	Since $\{z\in\Omega: G_{\Omega}(z,z_0)=a_0\}\subset U\Subset\Omega$, there exists $a_1>a_0$ satisfying 
	$$[a_0,a_1]\cap\{a<0:\exists z\in\partial U\,s.t.\,G_{\Omega}(z,z_0)=a\}=\emptyset,$$
	which shows that 
	$$\{z\in\Omega:a_0<G_{\Omega}(\cdot,z_0)<a_1\}\cap U\Subset U.$$
	Then we have
	\begin{displaymath}
		\begin{split}
		&\overline{\tilde U\cap\{z\in\Omega:G_{\Omega}(z,z_0)<a_1\}}\\
		=&\{z\in\Omega: G_{\Omega}(z,z_0)\le a_0\}\cup \overline{U\cap\{z\in\Omega:a_0<G_{\Omega}(z,z_0)<a_1\}}\\
		\subset&\tilde U	
		\end{split}
	\end{displaymath}
	As $\tilde U$ and $\{z\in\Omega: G_{\Omega}(z,z_0)<a_1\}$ are  connected open sets, we have $\{z\in\Omega: G_{\Omega}(z,z_0)<a_1\}\subset\tilde U$. Thus, we get $\{z\in\Omega:a_0<G_{\Omega}(\cdot,z_0)<a_1\}=\{z\in\Omega:a_0<G_{\Omega}(\cdot,z_0)<a_1\}\cap U\Subset U$.
\end{proof}

We recall the  coarea formula.
\begin{lemma}[see \cite{federer}]
	\label{l:coarea}Suppose that $\Omega$ is an open set in $\mathbb{R}^n$ and $u\in C^1(\Omega)$. Then for any $g\in L^1(\Omega)$, 
	$$\int_{\Omega}g(x)|\bigtriangledown u(x)|dx=\int_{\mathbb{R}}\left(\int_{u^{-1}(t)}g(x)dH_{n-1}(x)\right)dt,$$
	where $H_{n-1}$ is the $(n-1)$-dimensional Hausdorff measure.
\end{lemma}

Assume that $\Omega\subset\mathbb{C}$ is a domain bounded by finite analytic curves. Denote that 
$$s(t):=e^{2t}\lambda(\{G_{\Omega}(\cdot,z_0)<-t\})$$
 for $t\ge0$,
	where $\lambda(A)$ is the Lebesgue measure for any measurable subset $A$ of $\mathbb{C}$.

\begin{lemma}
	[\cite{BZ15}]\label{l:BZ}
	The function $s(t)$ is decreasing on $[0,+\infty)$.
\end{lemma}

The following lemma gives a characterization for $s(t)$ degenerating to a constant function.

\begin{lemma}
	\label{l:BZ2}If the function $s(t)$ is a constant function on $[0+\infty)$, then $\Omega$ is a disc.
\end{lemma}
\begin{proof}
	Firstly, we recall the proof of Lemma \ref{l:BZ} in \cite{BZ15}.
	
	Denote that $h(t)=\lambda(\{G_{\Omega}(\cdot,z_0)<-t\}).$ $f(t):=\log s(t)=2t+\log h(t)$. Let $-t$ be a regular value of $G_{\Omega}(\cdot,z_0)$, then we have 
	$$f'(t)=2+\frac{h'(t)}{h(t)}.$$
	Lemma \ref{l:coarea} shows that 
	$$h(t)=\int_{-\infty}^{-t}\int_{\{G_{\Omega}(\cdot,z_0)=s\}}\frac{d\sigma}{|\nabla G|}ds,$$
	which implies that
	\begin{equation}\nonumber
		h'(t)=-\int_{\{G_{\Omega}(\cdot,z_0)=-t\}}\frac{d\sigma}{|\nabla G|}.
	\end{equation}
	Using Cauchy-Schwarz inequality, we have
	$$-h'(t)\ge \frac{(\sigma(\{G_{\Omega}(\cdot,z_0)=-t\}))^2}{\int_{\{G_{\Omega}(\cdot,z_0)=-t\}}|\nabla G|d\sigma}=\frac{(\sigma(\{G_{\Omega}(\cdot,z_0)=-t\}))^2}{2\pi}.$$
	The isoperimetric inequality shows that 
	\begin{equation}
		\label{eq:1231a}(\sigma(\{G_{\Omega}(\cdot,z_0)=-t\}))^2\ge 4\pi\lambda(\{G_{\Omega}(\cdot,z_0)<-t\}),
	\end{equation}
	then we have $f'(t)\le 0$.
	
	Now, we prove Lemma \ref{l:BZ2}.
	If $s(t)$ is a constant function, then $f'\equiv0$, which implies that inequality \eqref{eq:1231a} becomes an equality. Following from the characterization of the isoperimetric inequality becoming an equality, we know that $\{G_{\Omega}(\cdot,z_0)<-t\}$ is a disc for any regular value $-t$ of $G_{\Omega}(\cdot,z_0)$. Thus, $\Omega$ is a disc.
\end{proof}

\subsection{Other useful lemmas}In this section, we give some lemmas, which will be used in the proofs of the main theorems.

Let $\theta_0\in(0,\pi]$, and denote that $L_{\theta_0}:=\{z\in\Delta:Arg(z)\in(0,\theta_0)\}$, where $Arg(z)=\theta\in(-\pi,\pi]$ for any $z=re^{i\theta}\in\mathbb{C}$.

\begin{lemma}
	\label{l:zero point}
	Let $v$ be a subharmonic function on the unit disc $\Delta$ satisfying that $v=k\log|z|$ on $L_{\theta_0}$, where $k>0$ is a constant. Then $v-k\log|z|$ is subharmonic on $\Delta$. 
\end{lemma}
\begin{proof}
 It suffices to prove the Lelong number $v(dd^cv,o)\ge k$, and we prove it by contradiction: if not, we have $v(dd^c\tilde v,o)=0$, where $\tilde v:=v-v(dd^cv,o)\log |z|$ is a subharmonic function on $\Delta$. Denote that $C:=\sup_{z\in \Delta_{\frac{1}{2}}}\tilde v<+\infty$.
We have 
 \begin{equation}
 	\nonumber
 	\begin{split}
 		v(dd^c\tilde v,o)&=\lim_{r\rightarrow 0}\frac{\frac{1}{2\pi }\int_{0}^{2\pi}\tilde v(re^{i\theta})d\theta}{\log r}\\
 		&\ge\lim_{r\rightarrow0}\frac{C+\frac{1}{2\pi}\int_{0}^{\theta_0}(k-v(dd^cv,o))\log|r|d\theta}{\log r}\\
 		&=\frac{(k-v(dd^cv,o))\theta_0}{2\pi}>0,
 	\end{split}
 \end{equation}
which contradicts to $v(dd^c\tilde v,o)=0$, hence $v-k\log|z|$ is subharmonic on $\Delta$.
\end{proof}

\begin{lemma}
	\label{l:harmonic subharmonic}
	Let $v$ be a subharmonic function on the unit disc $\Delta$ satisfying that $v=0$ on $L_{\theta_0}$. Then $v(o)=0$.
\end{lemma}
\begin{proof}
	As $v$ is upper semicontinuous, $v(o)\ge0$. As $v$ is subharmonic, it follows from the mean value inequality that 
	\begin{equation}
		\nonumber
		v(o)\le\frac{1}{2\pi}\int_{0}^{2\pi}v(re^{i\theta})d\theta=\frac{1}{2\pi}\int_{\theta_0}^{2\pi}v(re^{i\theta})d\theta
	\end{equation}
holds for any $r>0$. Then there is a point $z_r\in\{|z|=r\}$ such that $$v(z_r)\ge\frac{2\pi}{2\pi-\theta_0}v(o)$$ for any $r$. As $v$ is upper semicontinuous, we have
$$v(o)\ge\limsup_{r\rightarrow0}v(z_r)\ge\limsup_{r\rightarrow0}\frac{2\pi}{2\pi-\theta_0}v(o)=\frac{2\pi}{2\pi-\theta_0}v(o),$$
which shows that $v(o)=0$.
\end{proof}

Let us give two lemmas on concave functions.

\begin{lemma}\label{l:concave}
    Suppose $x(t): \mathbb{R}^+\to\mathbb{R}^+$ is a strictly decreasing function. If $-\log x(t)$ is convex, and $\log x(t)+t$ is lower bounded on $\mathbb{R}^+$, then $x(-\log r)$ is concave with respect to $r\in (0,1)$.
\end{lemma}

\begin{proof}
    We may assume $x(t)\in C^2(0,+\infty)$, otherwise we use the approximation method. Then $x'(t)<0$. By that $-\log x(t)$ is convex, direct computation gives
    \begin{equation}\label{x''x-x'2}
        x''(t)x(t)-(x'(t))^2\le 0.
    \end{equation}
    Since $-\log x(t)$ is convex and $\log x(t)+t$ is lower bounded, we have that $\log x(t)+t$ is increasing on $\mathbb{R}^+$, which implies that
    \begin{equation}\label{x'+x}
        x'(t)+x(t)\ge 0.
    \end{equation}
    Combining (\ref{x''x-x'2}) with (\ref{x'+x}), we get $x''(t)+x'(t)\le 0$, yielding that $x(-\log r)$ is concave with respect to $r\in (0,1)$.
\end{proof}

\begin{lemma}
	\label{l:notconvex}
	Let $g(r)$ be a concave function on $[0,1]$, which is strictly increasing on $[0,1]$ and satisfies $g(0)=0$. If there exists $r_0\in(0,1)$ such that  $g(r)=ar+b$ on $[r_0,1]$ and $g(r)\not\equiv ar$ on $[0,1]$, where $a,b\in\mathbb{R}$, then $h(t):=-\log g(e^{-t})$ is not a convex function on $[0,+\infty)$. 
\end{lemma}
\begin{proof}
	It is clear that $h(t)=-\log(ae^{-t}+b)$ on $[0,-\log r_0]$. Then we have 
	\begin{equation}
		\label{eq:230101}h''=\frac{-abe^{t}}{(a+be^t)^2}
	\end{equation}
	on $[0,-\log r_0]$.
	As $g(r)$ is strictly increasing and concave on $[0,1]$ and $g(r)\not\equiv ar$ on $[0,1]$, then $a>0$ and $b>0$. Inequality \eqref{eq:230101} shows that $h''<0$ on $[0,-\log r_0]$. Thus, $h(t)$ is not a convex function on $[0,+\infty)$.
\end{proof}

\section{Proofs of Theorem \ref{partiallylinear}, Remark \ref{tildec} and Theorem \ref{s-con}}

In this section, we prove Theorem \ref{partiallylinear}, Remark \ref{tildec} and Theorem \ref{s-con}.

\begin{proof}[Proof of Theorem \ref{partiallylinear}]
	 Firstly, we prove that there exists a unique holomorphic $(n,0)$ form $F$ on $\{\psi<-T_1\}$ satisfying $(F-f)\in H^0(Z_0,(\mathcal{O}(K_M)\otimes\mathcal{F})|_{Z_0})$, and $G(t;c)=\int_{\{\psi<-t\}}|F|^2e^{-\varphi}c(-\psi)$ for any $t\in [T_1,T_2]$. According to the assumptions and Lemma \ref{G(t)=0}, we can assume that $G(T_1)\in(0,+\infty)$.
	 
	 For any $t\in [T_1,T_2]$, by Lemma \ref{F_t}, there exists a unique holomorphic $(n,0)$ form $F_t$ on $\{\psi<-t\}$ such that $(F_t-f)\in H^0(Z_0,(\mathcal{O}(K_M)\otimes\mathcal{F})|_{Z_0})$ and
	 \begin{equation}\nonumber
	 	\int_{\{\psi<-t\}}|F_t|^2e^{-\varphi}c(-\psi)=G(t).
	 \end{equation}
 	 And by Lemma \ref{dbarequa}, for any $t\in [T_1,T_2)$ and $B_j>0$, there exists a holomorphic $(n,0)$ form $\tilde{F}_j$ on $\{\psi<-T_1\}$ such that $(\tilde{F}_j-F_t)\in H^0(Z_0,(\mathcal{O}(K_M)\otimes\mathcal{I}(\varphi+\psi))|_{Z_0})\subset H^0(Z_0,(\mathcal{O}(K_M)\otimes\mathcal{F})|_{Z_0})$ and
	 \begin{flalign}\label{p-pl-1}
	 	\begin{split}	 			
	 		&\int_{\{\psi<-T_1\}}|\tilde{F}_j-(1-b_{t,B_j}(\psi))F_t|^2e^{-\varphi-\psi+v_{t,B_j}(\psi)}c(-v_{t,B_j}(\psi))\\
	 	\leq&\int_{T_1}^{t+B_j}c(s)e^{-s}\mathrm{d}s\int_{\{\psi<-T_1\}}\frac{1}{B_j}\mathbb{I}_{\{-t-B_j<\psi<-t\}}|F_t|^2e^{-\varphi-\psi}.
	 	\end{split}
	 \end{flalign}
	We denote $v_{t,B_j}$ by $v_j$. As $s\leq v_j(s)$ and $c(s)e^{-s}$ is decreasing, we have
	\begin{equation}\nonumber
		e^{-\psi+v_j(\psi)}c(-v_j(\psi))\geq c(-\psi).
	\end{equation}
	Combining with (\ref{p-pl-1}), we obtain
	\begin{flalign}\label{p-pl-2}
		\begin{split}
			&\int_{\{\psi<-T_1\}}|\tilde{F}_j-(1-b_{t,B_j}(\psi))F_t|^2e^{-\varphi}c(-\psi)\\
			\leq&\int_{\{\psi<-T_1\}}|\tilde{F}_j-(1-b_{t,B_j}(\psi))F_t|^2e^{-\varphi-\psi+v_j(\psi)}c(-v_j(\psi))\\
			\leq&\int_{T_1}^{t+B_j}c(s)e^{-s}\mathrm{d}s\int_{\{\psi<-T_1\}}\frac{1}{B_j}\mathbb{I}_{\{-t-B_j<\psi<-t\}}|F_t|^2e^{-\varphi-\psi}\\
			\leq&\frac{e^{t+B_j}\int_{T_1}^{t+B_j}c(s)e^{-s}\mathrm{d}s}{\inf_{s\in(t,t+B_j)}c(s)}\int_{\{\psi<-T_1\}}\frac{1}{B_j}\mathbb{I}_{\{-t-B_j<\psi<-t\}}|F_t|^2e^{-\varphi}c(-\psi)\\
			\leq&\frac{e^{t+B_j}\int_{T_1}^{t+B_j}c(s)e^{-s}\mathrm{d}s}{\inf_{s\in(t,t+B_j)}c(s)}\left(\int_{\{\psi<-t\}}\frac{1}{B_j}|F_t|^2e^{-\varphi}c(-\psi)-\int_{\{\psi<-t-B_j\}}\frac{1}{B_j}|F_t|^2e^{-\varphi}c(-\psi)\right)\\
			\leq&\frac{e^{t+B_j}\int_{T_1}^{t+B_j}c(s)e^{-s}\mathrm{d}s}{\inf_{s\in(t,t+B_j)}c(s)}\cdot \frac{G(t)-G(t+B_j)}{B_j}.
		\end{split}
	\end{flalign}
	We may assume $B_j\in (0,T_2-t)$ and $\lim\limits_{j\rightarrow +\infty}B_j=0$. Then by the partially linear assumption of $G(h^{-1}(r))$, (\ref{p-pl-2}) implies that there is a nonnegative constant $\kappa$ such that
	\begin{flalign}\label{p-pl-3}
		\begin{split}
		&\int_{\{\psi<-T_1\}}|\tilde{F}_j-(1-b_{t,B_j}(\psi))F_t|^2e^{-\varphi}c(-\psi)\\
		\leq&\frac{e^{t+B_j}\int_{T_1}^{t+B_j}c(s)e^{-s}\mathrm{d}s}{\inf_{s\in(t,t+B_j)}c(s)}\cdot \frac{\kappa\int_t^{t+B_j}c(s)e^{-s}\mathrm{d}s}{B_j}.
		\end{split}
	\end{flalign}
	Here $\kappa$ is the slope of the linear function $G(h^{-1}(r))|_{[h(T_2),h(T_1)]}$.
	
	We prove that $\int_{\{\psi<-T_1\}}|\tilde{F}_j|^2e^{-\varphi}c(-\psi)$ is uniformly bounded with respect to $j$. Note that
	\begin{flalign}\nonumber
		&\left(\int_{\{\psi<-T_1\}}|\tilde{F}_j-(1-b_{t,B_j}(\psi))F_t|^2e^{-\varphi}c(-\psi)\right)^{1/2}\\
		\geq&\nonumber \left(\int_{\{\psi<-T_1\}}|\tilde{F}_j|^2e^{-\varphi}c(-\psi)\right)^{1/2}-\left(\int_{\{\psi<-T_1\}}|(1-b_{t,B_j}(\psi))F_t|^2e^{-\varphi}c(-\psi)\right)^{1/2},
	\end{flalign}
	then it follows from (\ref{p-pl-3}) that
	\begin{flalign}\nonumber
		&\left(\int_{\{\psi<-T_1\}}|\tilde{F}_j|^2e^{-\varphi}c(-\psi)\right)^{1/2}\\
		\leq&\left(\frac{e^{t+B_j}\int_{T_1}^{t+B_j}c(s)e^{-s}\mathrm{d}s}{\inf_{s\in(t,t+B_j)}c(s)}\cdot \frac{\kappa\int_t^{t+B_j}c(s)e^{-s}\mathrm{d}s}{B_j}\right)^{1/2}\nonumber \\
		&\nonumber+\left(\int_{\{\psi<-T_1\}}|(1-b_{t,B_j}(\psi))F_t|^2e^{-\varphi}c(-\psi)\right)^{1/2}.
	\end{flalign}
	Since $0\leq b_{t,B_j}(\psi)\leq 1$, $\int_{\{\psi<-T_1\}}|\tilde{F}_j|^2e^{-\varphi}c(-\psi)$ is uniformly bounded for any $j$. Then it follows from Lemma \ref{s_K} that there is a subsequence of $\{\tilde{F}_j\}$, denoted by $\{\tilde{F}_{j_k}\}$, which is uniformly convergent to a holomorphic $(n,0)$ form $\tilde{F}$ on $\{\psi<-T_1\}$ on any compact subset of $\{\psi<-T_1\}$. Then by Fatou's Lemma,
	\begin{equation}\nonumber
		\int_{\{\psi<-T_1\}}|\tilde{F}|^2e^{-\varphi}c(-\psi)\leq\liminf_{k\rightarrow +\infty}\int_{\{\psi<-T_1\}}|\tilde{F}_{j_k}|^2e^{-\varphi}c(-\psi)<+\infty.
	\end{equation} 
	As $(\tilde{F}_j-F_t)\in H^0(Z_0,(\mathcal{O}(K_M)\otimes\mathcal{F})|_{Z_0})$ for any $j$, Lemma \ref{module} implies that $(\tilde{F}-F_t)\in H^0(Z_0,(\mathcal{O}(K_M)\otimes\mathcal{F})|_{Z_0})$. With direct calculations, we have
	\begin{equation}\nonumber
		\lim_{j\rightarrow +\infty}b_{t,B_j}(s)=\left\{
		\begin{array}{ll}
			0, & s \in (-\infty,-t) \\
			1, & s \in [-t,+\infty)
		\end{array}
		\right. ,
	\end{equation}
	and
		\begin{equation}\nonumber
		\lim_{j\rightarrow +\infty}v_{t,B_j}(s)=\left\{
		\begin{array}{ll}
			-t, & s \in (-\infty,-t) \\
			s, & s \in [-t,+\infty)
		\end{array}
		\right. .
	\end{equation}
	Then it follows from (\ref{p-pl-3}) and Fatou's Lemma that
	\begin{flalign}\nonumber
		\begin{split}
			&\int_{\{\psi<-t\}}|\tilde{F}-F_t|^2e^{-\varphi-\psi-t}c(t)+\int_{\{-t\leq\psi<-T_1\}}|\tilde{F}|^2e^{-\varphi}c(-\psi)\\
			=&\int_{\{\psi<-T_1\}}\lim_{k\rightarrow +\infty}|\tilde{F}_{j_k}-(1-b_{t,B_{j_k}}(\psi))F_t|^2e^{-\varphi-\psi+v_{t,B_{j_k}}(\psi)}c(-v_{t,B_{j_k}}(\psi))\\
			\leq&\liminf_{k\rightarrow +\infty}\int_{\{\psi<-T_1\}}|\tilde{F}_{j_k}-(1-b_{t,B_{j_k}}(\psi))F_t|^2e^{-\varphi-\psi+v_{t,B_{j_k}}(\psi)}c(-v_{t,B_{j_k}}(\psi))\\
			\leq&\liminf_{k\rightarrow +\infty}\left(\frac{e^{t+B_{j_k}}\int_{T_1}^{t+B_{j_k}}c(s)e^{-s}\mathrm{d}s}{\inf_{s\in(t,t+B_{j_k})}c(s)}\cdot \frac{\kappa\int_t^{t+B_{j_k}}c(s)e^{-s}\mathrm{d}s}{B_{j_k}}\right)\\
			=&\kappa\int_{T_1}^tc(s)e^{-s}\mathrm{d}s.
		\end{split}
	\end{flalign}
	It follows from the choice of $F_t$, and Lemma \ref{F_t} that
	\begin{flalign}\label{p-pl-5}
		\begin{split}
			&\int_{\{\psi<-T_1\}}|\tilde{F}|^2e^{-\varphi}c(-\psi)\\
			=&\int_{\{\psi<-T_1\}}|\tilde{F}|^2e^{-\varphi}c(-\psi)+\int_{\{-t\leq\psi<-T_1\}}|\tilde{F}|^2e^{-\varphi}c(-\psi)\\
			=&\int_{\{\psi<-t\}}|F_t|^2e^{-\varphi}c(-\psi)+\int_{\{\psi<-t\}}|\tilde{F}-F_t|^2e^{-\varphi}c(-\psi)\\
			&+\int_{\{-t\leq\psi<-T_1\}}|\tilde{F}|^2e^{-\varphi}c(-\psi).
		\end{split}
	\end{flalign}
	As $c(t)e^{-t}$ is decreasing, we have $e^{\psi}c(-\psi)\leq e^{-t}c(t)$ on $\{\psi<-t\}$ and
	\begin{flalign}\label{p-pl-6}
		\begin{split}
			\int_{\{\psi<-t\}}|\tilde{F}-F_t|^2e^{-\varphi}c(-\psi)\leq\int_{\{\psi<-t\}}|\tilde{F}-F_t|^2e^{-\varphi-\psi-t}c(t).
		\end{split}
	\end{flalign}
	Combining (\ref{p-pl-6}) with (\ref{p-pl-5}), we get
		\begin{flalign}\label{p-pl-7}
		\begin{split}
			&\int_{\{\psi<-T_1\}}|\tilde{F}|^2e^{-\varphi}c(-\psi)\\
			=&\int_{\{\psi<-t\}}|F_t|^2e^{-\varphi}c(-\psi)+\int_{\{\psi<-t\}}|\tilde{F}-F_t|^2e^{-\varphi}c(-\psi)\\
			&+\int_{\{-t\leq\psi<-T_1\}}|\tilde{F}|^2e^{-\varphi}c(-\psi)\\
			\leq&\int_{\{\psi<-t\}}|F_t|^2e^{-\varphi}c(-\psi)+\int_{\{\psi<-t\}}|\tilde{F}-F_t|^2e^{-\varphi-\psi-t}c(-t)\\
			&+\int_{\{-t\leq\psi<-T_1\}}|\tilde{F}|^2e^{-\varphi}c(-\psi)\\
			\leq& G(t)+\kappa\int_{T_1}^tc(s)e^{-s}\mathrm{d}s\\
			=&G(T_1).
		\end{split}
	\end{flalign}
	Since $(\tilde{F}-F_t)\in H^0(Z_0,(\mathcal{O}(K_M)\otimes\mathcal{F})|_{Z_0})$, by Lemma \ref{F_t}, $\tilde{F}$ is exactly the unique holomorphic $(n,0)$ form such that $(\tilde{F}-F_t)\in H^0(Z_0,(\mathcal{O}(K_M)\otimes\mathcal{F})|_{Z_0})$ and
	\begin{equation}\nonumber
		\int_{\{\psi<-T_1\}}|\tilde{F}|^2e^{-\varphi}c(-\psi)=G(T_1).
	\end{equation}
	Or we can say $\tilde{F}=F_{T_1}$.
	
	Now we prove that $\tilde{F}=F_t$ on $\{\psi<-t\}$. By the above discussions, the inequality (\ref{p-pl-7}) should be equality. Then the inequality (\ref{p-pl-6}) should also be equality. It means that
		\begin{equation}\label{p-pl-8}
			\int_{\{\psi<-t\}}|\tilde{F}-F_t|^2e^{-\varphi-\psi}(e^{-t}c(t)-e^{\psi}c(-\psi))=0.
	\end{equation}
	Note that $\int_{T_1}^{+\infty}c(t)e^{-t}\mathrm{d}t<+\infty$. Thus we can find some $t'>t$ and $\varepsilon>0$ such that
	\begin{equation}\nonumber
		e^{-t}c(t)-e^{\psi}c(-\psi)>\varepsilon>0, \ \forall z\in \{\psi<-t'\}.
	\end{equation}
	Then (\ref{p-pl-8}) implies that
	\begin{equation}\nonumber
	\varepsilon\int_{\{\psi<-t'\}}|\tilde{F}-F_t|^2e^{-\varphi-\psi}=0.
\end{equation}
	Since $\tilde{F}$ and $F_t$ are holomorphic $(n,0)$ forms on $\{\psi<-t\}$, we have $\tilde{F}=F_t$, i.e. $F_{T_1}=F_t$ on $\{\psi<-t\}$. This is true for any $t\in [T_1, T_2]$.
	
	Finally, according to the Lebesgue dominated convergence theorem and the assumptions, we have
	\begin{equation}\nonumber
		\int_{\{\psi<T_2\}}|F_{T_1}|^2e^{-\varphi}c(-\psi)=\lim_{t\rightarrow T_2-0}G(t)=G(T_2).
	\end{equation}
	From this we get $F_{T_1}=F_{T_2}$ on $\{\psi<-T_2\}$. Then we have proved that $F_{T_1}$ is exactly the unique holomorphic $(n,0)$ form $F$ on $\{\psi<-T_1\}$ satisfying $(F-f)\in H^0(Z_0,(\mathcal{O}(K_M)\otimes\mathcal{F})|_{Z_0})$, and $G(t;c)=\int_{\{\psi<-t\}}|F|^2e^{-\varphi}c(-\psi)$ for any $t\in [T_1,T_2]$.	
	
	Now we prove the rest result of Theorem \ref{partiallylinear}.
	
	As $a(t)$ is a nonnegative measurable function on $[T_1,T_2]$, we can find a sequence of functions $\{\sum_{j=1}^{n_i}a_{ij}\mathbb{I}_{E_ij}\}_{i\in\mathbb{N}_+}$ ($n_i<+\infty$ for any $i\in\mathbb{N}_+$) satisfying that the sequence is increasingly convergent to $a(t)$ for a.e. $t\in [T_1,T_2]$, where $E_{ij}$ is a Lebesgue measurable subset of $[T_1,T_2]$ and $a_{ij}>0$ is a constant for any $i,j$. It follows from Levi's Theorem that we only need to prove the equality (\ref{a(t)}) under the assumption $a(t)=\mathbb{I}_E(t)$, where $E$ is a Lebesgue measurable subset of $[T_1,T_2]$.
	
	Note that
	\begin{equation}\nonumber
		G(t)=\int_{\{\psi<-t\}}|F|^2e^{-\varphi}c(-\psi)=G(T_2)+\frac{G(T_1)-G(T_2)}{\int_{T_1}^{T_2}c(s)e^{-s}\mathrm{d}s}\int_t^{T_2}c(s)e^{-s}\mathrm{d}s
	\end{equation}
	for any $t\in [T_1,T_2]$. Then
	\begin{equation}\label{t1t2}
		\int_{\{-t_2\leq\psi<-t_1\}}|F|^2e^{-\varphi}c(-\psi)=\frac{G(T_1)-G(T_2)}{\int_{T_1}^{T_2}c(s)e^{-s}\mathrm{d}s}\int_{t_1}^{t_2}c(s)e^{-s}\mathrm{d}s
	\end{equation}
	holds for any $T_1\leq t_1<t_2\leq T_2$. Then for any Lebesgue zero measure subset $N$ of $[T_1,T_2]$, we get
	\begin{equation}\label{zerom}
		\int_{\{-\psi(z)\in N\}}|F|^2e^{-\varphi}=0
	\end{equation}
	from Lebesgue dominated convergence theorem. Since $c(t)e^{-t}$ is decreasing on $[T_1,T_2]$, we can find a countable subsets $\{s_j\}_{j\in\mathbb{N}_+}\subset[T_1,T_2]$ such that $c(t)$ is continuous besides $\{s_j\}$. It means that there exists a sequence of open sets $\{U_k\}$ such that $\{s_j\}\subset U_k\subset [T_1,T_2]$ and $\lim\limits_{k\rightarrow +\infty}\mu(U_k)=0$, where $\mu$ is the Lebesgue measure on $\mathbb{R}$. Then for any $[t_1',t_2']\subset [T_1,T_2]$, we have
	\begin{flalign}\nonumber
		\begin{split}
			&\int_{\{-t_2'\leq\psi<-t_1'\}}|F|^2e^{-\varphi}\\
			=&\int_{\{-\psi(z)\in(t_1',t_2']\setminus U_k\}}|F|^2e^{-\varphi}+\int_{\{-\psi(z)\in(t_1',t_2']\cup U_k\}}|F|^2e^{-\varphi}\\
			=&\lim_{n\rightarrow+\infty}\sum_{i=1}^{n-1}\int_{\{-\psi(z)\in I_{i,n}\setminus U_k\}}|F|^2e^{-\varphi}+\int_{\{-\psi(z)\in(t_1',t_2']\cup U_k\}}|F|^2e^{-\varphi},
		\end{split}
	\end{flalign}
	where $I_{i,n}=(t_1'+ia_n,t_1'+(i+1)a_n]$, and $a_n=(t_2'-t_1')/n$. Using equality (\ref{t1t2}), we have
	\begin{flalign}\nonumber
		\begin{split}
			&\lim_{n\rightarrow+\infty}\sum_{i=1}^{n-1}\int_{\{-\psi(z)\in I_{i,n}\setminus U_k\}}|F|^2e^{-\varphi}\\
			\leq&\limsup_{n\rightarrow+\infty}\sum_{i=1}^{n-1}\frac{1}{\inf_{I_{i,n}\setminus U_k}c(t)}\int_{\{-\psi(z)\in I_{i,n}\setminus U_k\}}|F|^2e^{-\varphi}c(-\psi)\\
			\leq&\frac{G(T_1)-G(T_2)}{\int_{T_1}^{T_2}c(s)e^{-s}\mathrm{d}s}\limsup_{n\rightarrow+\infty}\sum_{i=1}^{n-1}\frac{1}{\inf_{I_{i,n}\setminus U_k}c(t)}\int_{I_{i,n}\setminus U_k}c(s)e^{-s}\mathrm{d}s.
		\end{split}
	\end{flalign}
	In addition, according to the choice of $U_k$, we have
	\begin{flalign}\nonumber
		\begin{split}
			&\limsup_{n\rightarrow+\infty}\sum_{i=1}^{n-1}\frac{1}{\inf_{I_{i,n}\setminus U_k}c(t)}\int_{I_{i,n}\setminus U_k}c(s)e^{-s}\mathrm{d}s\\
			\leq&\limsup_{n\rightarrow+\infty}\sum_{i=1}^{n-1}\frac{\sup_{I_{i,n}\setminus U_k}c(t)}{\inf_{I_{i,n}\setminus U_k}c(t)}\int_{I_{i,n}\setminus U_k}c(s)e^{-s}\mathrm{d}s\\
			=&\int_{(t_1',t_2]\setminus U_k}e^{-s}\mathrm{d}s.
		\end{split}
	\end{flalign}
	Now we can obtain that
	\begin{flalign}\nonumber
		\begin{split}
			&\int_{\{-t_2'\leq\psi<-t_1'\}}|F|^2e^{-\varphi}\\
			=&\int_{\{-\psi(z)\in(t_1',t_2']\setminus U_k\}}|F|^2e^{-\varphi}+\int_{\{-\psi(z)\in(t_1',t_2']\cup U_k\}}|F|^2e^{-\varphi}\\
			\leq&\frac{G(T_1)-G(T_2)}{\int_{T_1}^{T_2}c(s)e^{-s}\mathrm{d}s}\int_{(t_1',t_2]\setminus U_k}e^{-s}\mathrm{d}s+\int_{\{-\psi(z)\in(t_1',t_2']\cup U_k\}}|F|^2e^{-\varphi}.
		\end{split}
	\end{flalign}
	Let $k\rightarrow+\infty$, it follows from equality (\ref{zerom}) that
	\begin{equation}\nonumber
		\int_{\{-t_2'\leq\psi<-t_1'\}}|F|^2e^{-\varphi}\leq\frac{G(T_1)-G(T_2)}{\int_{T_1}^{T_2}c(s)e^{-s}\mathrm{d}s}\int_{t_1'}^{t_2'}e^{-s}\mathrm{d}s.
	\end{equation}
	Using the same methods, we can also get that
	\begin{equation}\nonumber
		\int_{\{-t_2'\leq\psi<-t_1'\}}|F|^2e^{-\varphi}\geq\frac{G(T_1)-G(T_2)}{\int_{T_1}^{T_2}c(s)e^{-s}\mathrm{d}s}\int_{t_1'}^{t_2'}e^{-s}\mathrm{d}s.
	\end{equation}
	Then we know that
	\begin{equation}\nonumber
		\int_{\{-t_2'\leq\psi<-t_1'\}}|F|^2e^{-\varphi}=\frac{G(T_1)-G(T_2)}{\int_{T_1}^{T_2}c(s)e^{-s}\mathrm{d}s}\int_{t_1'}^{t_2'}e^{-s}\mathrm{d}s.
	\end{equation}
	Thus for any open subset $U\subset [T_1,T_2]$ and any compact subset $K\subset [T_1,T_2]$, we have
	\begin{equation}\nonumber
		\int_{\{-\psi(z)\in U\}}|F|^2e^{-\varphi}=\frac{G(T_1)-G(T_2)}{\int_{T_1}^{T_2}c(s)e^{-s}\mathrm{d}s}\int_Ue^{-s}\mathrm{d}s,
	\end{equation}
	and
	\begin{equation}\nonumber
		\int_{\{-\psi(z)\in K\}}|F|^2e^{-\varphi}=\frac{G(T_1)-G(T_2)}{\int_{T_1}^{T_2}c(s)e^{-s}\mathrm{d}s}\int_Ke^{-s}\mathrm{d}s.
	\end{equation}
	 Then for the Lebesgue measurable subset $E\subset [T_1,T_2]$, and $t_1,t_2\in [T_1,T_2]$ with $t_1<t_2$, we can find a sequence of compact sets $\{K_j\}$ and a sequence of open sets $\{U_j\}$ such that 
	 \begin{equation}\nonumber
	 	K_1\subset\ldots\subset K_j\subset K_{j+1}\subset\ldots E\cap (t_1,t_2]\subset\ldots\subset U_{j+1}\subset U_j\subset\ldots\subset U_1\subset [T_1,T_2],
	 \end{equation}
	 and $\lim\limits_{j\rightarrow +\infty}\mu(U_j\setminus K_j)=0$. Then we have
	 \begin{flalign}\nonumber
	 	\begin{split}
	 	 &\int_{\{-t_2\leq\psi<-t_1\}}|F|^2e^{-\varphi}\mathbb{I}_E(-\psi)\\
	 	 =&\int_{\{-\psi\in E\cap (t_1,t_2]\}}|F|^2e^{-\varphi}\\
	 	 \leq&\liminf_{j\rightarrow +\infty}\int_{\{-\psi\in U_j\cap (t_1,t_2]\}}|F|^2e^{-\varphi}\\
	 	 \leq&\liminf_{j\rightarrow +\infty}\frac{G(T_1)-G(T_2)}{\int_{T_1}^{T_2}c(s)e^{-s}\mathrm{d}s}\int_{\{-\psi\in U_j\cap (t_1,t_2]\}}e^{-s}\mathrm{d}s\\
	 	 \leq&\frac{G(T_1)-G(T_2)}{\int_{T_1}^{T_2}c(s)e^{-s}\mathrm{d}s}\int_{\{-\psi\in E\cap (t_1,t_2]\}}e^{-s}\mathrm{d}s\\
	 	 =&\frac{G(T_1)-G(T_2)}{\int_{T_1}^{T_2}c(s)e^{-s}\mathrm{d}s}\int_{t_1}^{t_2}e^{-s}\mathbb{I}_E\mathrm{d}s,
	 	 \end{split}
	 \end{flalign}
	 and
	 \begin{flalign}\nonumber
	 	\begin{split}
	 		&\int_{\{-t_2\leq\psi<-t_1\}}|F|^2e^{-\varphi}\mathbb{I}_E(-\psi)\\
	 		=&\int_{\{-\psi\in E\cap (t_1,t_2]\}}|F|^2e^{-\varphi}\\
	 		\geq&\liminf_{j\rightarrow +\infty}\int_{\{-\psi\in K_j\cap (t_1,t_2]\}}|F|^2e^{-\varphi}\\
	 		\geq&\liminf_{j\rightarrow +\infty}\frac{G(T_1)-G(T_2)}{\int_{T_1}^{T_2}c(s)e^{-s}\mathrm{d}s}\int_{\{-\psi\in K_j\cap (t_1,t_2]\}}e^{-s}\mathrm{d}s\\
	 		\geq&\frac{G(T_1)-G(T_2)}{\int_{T_1}^{T_2}c(s)e^{-s}\mathrm{d}s}\int_{\{-\psi\in E\cap (t_1,t_2]\}}e^{-s}\mathrm{d}s\\
	 		=&\frac{G(T_1)-G(T_2)}{\int_{T_1}^{T_2}c(s)e^{-s}\mathrm{d}s}\int_{t_1}^{t_2}e^{-s}\mathbb{I}_E\mathrm{d}s.
	 	\end{split}
	 \end{flalign}
	 It means that
	 \begin{equation}\nonumber
	 	\int_{\{-t_2\leq\psi<-t_1\}}|F|^2e^{-\varphi}\mathbb{I}_E(-\psi)=\frac{G(T_1)-G(T_2)}{\int_{T_1}^{T_2}c(s)e^{-s}\mathrm{d}s}\int_{t_1}^{t_2}e^{-s}\mathbb{I}_E\mathrm{d}s,
	 \end{equation}
 	which implies that equality (\ref{a(t)}) holds.	
\end{proof}

Now we give the proof of Remark \ref{tildec}.
\begin{proof}[Proof of Remark \ref{tildec}]
	Firstly, we have $G(t;\tilde{c})\leq\int_{\{\psi<-t\}}|F|^2e^{-\varphi}\tilde{c}(-\psi)$, then we may assume that $G(t;\tilde{c})<+\infty$.

	Secondly, by Lemma \ref{F_t}, for any $t\in [T_1,T_2]$, there exists a unique holomorphic $(n,0)$ form $F_t$ on $\{\psi<-t\}$ satisfying $(F_t-f)\in H^0(Z_0,(\mathcal{O}(K_M)\otimes\mathcal{F})|_{Z_0})$ , and
	\begin{equation}\nonumber
		\int_{\{\psi<-t\}}|F_t|^2e^{-\varphi}\tilde{c}(-\psi)=G(t;\tilde{c}).
	\end{equation}
	According to Lemma \ref{F_t} and $F_t\in\mathcal{H}(\tilde{c},t)\subset\mathcal{H}(c,t)$, we have
	\begin{flalign}\nonumber
		\begin{split}
			&\int_{\{\psi<-t'\}}|F_t|^2e^{-\varphi}c(-\psi)\\
			=&\int_{\{\psi<-t'\}}|F|^2e^{-\varphi}c(-\psi)+\int_{\{\psi<-t'\}}|F_t-F|^2e^{-\varphi}c(-\psi)
		\end{split}
	\end{flalign}
	for any $t'\in [t,T_2]$. Then for any $t_1,t_2\in [t,T_2]$ with $t_1<t_2$,
	\begin{flalign}\label{|F_t-F|}
		\begin{split}
			&\int_{\{-t_2\leq\psi<-t_1\}}|F_t|^2e^{-\varphi}c(-\psi)\\
			=&\int_{\{-t_2\leq\psi<-t_1\}}|F|^2e^{-\varphi}c(-\psi)+\int_{\{-t_2\leq\psi<-t_1\}}|F_t-F|^2e^{-\varphi}c(-\psi).
		\end{split}
		\end{flalign} 
	Thus for any Lebesgue zero measure subset $N$ of $(t,T_2]$, it follows from equality (\ref{|F_t-F|}), equality (\ref{zerom}) and Lebesgue dominated convergence theorem that
	\begin{equation}\label{zerom2}
		\int_{\{-\psi(z)\in N\}}|F_t|^2e^{-\varphi}=\int_{\{-\psi(z)\in N\}}|F_t-F|^2e^{-\varphi}.
	\end{equation}
	Since $c(t)e^{-t}$ is decreasing on $[T_1,T_2]$, we can find a countable subsets $\{s_j\}_{j\in\mathbb{N}_+}\subset(t,T_2]$ such that $c(t)$ is continuous besides $\{s_j\}$. It means that there exists a sequence of open sets $\{U_k\}$ such that $\{s_j\}\subset U_k\subset [T_1,T_2]$ and $\lim\limits_{k\rightarrow +\infty}\mu(U_k)=0$, where $\mu$ is the Lebesgue measure on $\mathbb{R}$. Then for any $(t_1',t_2']\subset (t,T_2]$, we have
	\begin{flalign}\nonumber
		\begin{split}
			&\int_{\{-t_2'\leq\psi<-t_1'\}}|F_t|^2e^{-\varphi}\\
			=&\int_{\{-\psi(z)\in(t_1',t_2']\setminus U_k\}}|F_t|^2e^{-\varphi}+\int_{\{-\psi(z)\in(t_1',t_2']\cup U_k\}}|F_t|^2e^{-\varphi}\\
			=&\lim_{n\rightarrow+\infty}\sum_{i=1}^{n-1}\int_{\{-\psi(z)\in I_{i,n}\setminus U_k\}}|F_t|^2e^{-\varphi}+\int_{\{-\psi(z)\in(t_1',t_2']\cup U_k\}}|F_t|^2e^{-\varphi},
		\end{split}
	\end{flalign}
	where $I_{i,n}=(t_1'+ia_n,t_1'+(i+1)a_n]$, and $a_n=(t_2'-t_1')/n$. Using equality (\ref{|F_t-F|}), we have
	\begin{flalign}\nonumber
		\begin{split}
			&\lim_{n\rightarrow+\infty}\sum_{i=1}^{n-1}\int_{\{-\psi(z)\in I_{i,n}\setminus U_k\}}|F_t|^2e^{-\varphi}\\
			\leq&\limsup_{n\rightarrow+\infty}\sum_{i=1}^{n-1}\frac{1}{\inf_{I_{i,n}\setminus U_k}c(t)}\int_{\{-\psi(z)\in I_{i,n}\setminus U_k\}}|F_t|^2e^{-\varphi}c(-\psi)\\
			=&\limsup_{n\rightarrow+\infty}\sum_{i=1}^{n-1}\frac{1}{\inf_{I_{i,n}\setminus U_k}c(t)}(\int_{\{-\psi(z)\in I_{i,n}\setminus U_k\}}|F|^2e^{-\varphi}c(-\psi)\\
			&+\int_{\{-\psi(z)\in I_{i,n}\setminus U_k\}}|F_t-F|^2e^{-\varphi}c(-\psi)).
		\end{split}
	\end{flalign}
	In addition, according to the choice of $U_k$, we have
	\begin{flalign}\nonumber
		\begin{split}
			&\limsup_{n\rightarrow+\infty}\sum_{i=1}^{n-1}\frac{1}{\inf_{I_{i,n}\setminus U_k}c(t)}\int_{\{-\psi(z)\in I_{i,n}\setminus U_k\}}|F_t|^2e^{-\varphi}\\
			\leq&\limsup_{n\rightarrow+\infty}\sum_{i=1}^{n-1}\frac{1}{\inf_{I_{i,n}\setminus U_k}c(t)}(\int_{\{-\psi(z)\in I_{i,n}\setminus U_k\}}|F|^2e^{-\varphi}c(-\psi)\\
			+&\int_{\{-\psi(z)\in I_{i,n}\setminus U_k\}}|F_t-F|^2e^{-\varphi}c(-\psi))\\
			\leq&\limsup_{n\rightarrow+\infty}\sum_{i=1}^{n-1}\frac{\sup_{I_{i,n}\setminus U_k}c(t)}{\inf_{I_{i,n}\setminus U_k}c(t)}(\int_{\{-\psi(z)\in I_{i,n}\setminus U_k\}}|F|^2e^{-\varphi}\\
			+&\int_{\{-\psi(z)\in I_{i,n}\setminus U_k\}}|F_t-F|^2e^{-\varphi})\\
			&=\int_{\{-\psi(z)\in (t_1',t_2']\setminus U_k\}}|F|^2e^{-\varphi}+\int_{\{-\psi(z)\in (t_1',t_2']\setminus U_k\}}|F_t-F|^2e^{-\varphi}.
		\end{split}
	\end{flalign}
	Now we can obtain that
	\begin{flalign}\nonumber
		\begin{split}
			&\int_{\{-t_2'\leq\psi<-t_1'\}}|F_t|^2e^{-\varphi}\\
			=&\int_{\{-\psi(z)\in(t_1',t_2']\setminus U_k\}}|F_t|^2e^{-\varphi}+\int_{\{-\psi(z)\in(t_1',t_2']\cup U_k\}}|F_t|^2e^{-\varphi}\\
			\leq&\int_{\{-\psi(z)\in (t_1',t_2']\setminus U_k\}}|F|^2e^{-\varphi}+\int_{\{-\psi(z)\in (t_1',t_2']\setminus U_k\}}|F_t-F|^2e^{-\varphi}\\
			&+\int_{\{-\psi(z)\in(t_1',t_2']\cup U_k\}}|F_t|^2e^{-\varphi}.
		\end{split}
	\end{flalign}
	Let $k\rightarrow+\infty$, then it follows from equality (\ref{zerom2}) that
	\begin{flalign}\nonumber
		\begin{split}
		&\int_{\{-t_2'\leq\psi<-t_1'\}}|F|^2e^{-\varphi}\\
		\leq&\int_{\{-\psi(z)\in (t_1',t_2']\}}|F|^2e^{-\varphi}+\int_{\{-\psi(z)\in (t_1',t_2']\setminus N\}}|F_t-F|^2e^{-\varphi}\\
		&+\int_{\{-\psi(z)\in(t_1',t_2']\cup N\}}|F_t|^2e^{-\varphi},
		\end{split}
	\end{flalign}
	where $N=\bigcap_{k=1}^{+\infty}U_k$ and $\mu(N)=0$. Using the same methods, we can also get that
	\begin{flalign}\nonumber
	\begin{split}
		&\int_{\{-t_2'\leq\psi<-t_1'\}}|F|^2e^{-\varphi}\\
		\geq&\int_{\{-\psi(z)\in (t_1',t_2']\}}|F|^2e^{-\varphi}+\int_{\{-\psi(z)\in (t_1',t_2']\setminus N\}}|F_t-F|^2e^{-\varphi}\\
		&+\int_{\{-\psi(z)\in(t_1',t_2']\cup N\}}|F_t|^2e^{-\varphi}.
	\end{split}
\end{flalign}
	Then we know that
	\begin{flalign}\nonumber
	\begin{split}
		&\int_{\{-t_2'\leq\psi<-t_1'\}}|F|^2e^{-\varphi}\\
		=&\int_{\{-\psi(z)\in (t_1',t_2']\}}|F|^2e^{-\varphi}+\int_{\{-\psi(z)\in (t_1',t_2']\setminus N\}}|F_t-F|^2e^{-\varphi}\\
		&+\int_{\{-\psi(z)\in(t_1',t_2']\cup N\}}|F_t|^2e^{-\varphi}.
	\end{split}
\end{flalign}
	Thus for any open subset $U\subset (t,T_2]$ and any compact subset $K\subset (t,T_2]$, we have
	\begin{flalign}\nonumber
	\begin{split}
		&\int_{\{-\psi(z)\in U\}}|F|^2e^{-\varphi}\\
		=&\int_{\{-\psi(z)\in U\}}|F|^2e^{-\varphi}+\int_{\{-\psi(z)\in U\setminus N\}}|F_t-F|^2e^{-\varphi}\\
		&+\int_{\{-\psi(z)\in U\cup N\}}|F_t|^2e^{-\varphi},
	\end{split}
\end{flalign}
	and
	\begin{flalign}\nonumber
	\begin{split}
		&\int_{\{-\psi(z)\in K\}}|F|^2e^{-\varphi}\\
		=&\int_{\{-\psi(z)\in K\}}|F|^2e^{-\varphi}+\int_{\{-\psi(z)\in K\setminus N\}}|F_t-F|^2e^{-\varphi}\\
		&+\int_{\{-\psi(z)\in K\cup N\}}|F_t|^2e^{-\varphi}.
	\end{split}
\end{flalign}
	Then for the Lebesgue measurable subset $E\subset [T_1,T_2]$, we can find a sequence of compact sets $\{K_j\}$ such that
	\begin{equation}\nonumber
		K_1\subset\ldots\subset K_j\subset K_{j+1}\subset\ldots E\subset [T_1,T_2],
	\end{equation}
	and $\lim\limits_{j\rightarrow +\infty}\mu(E\setminus K_j)=0$. Thus we have
	\begin{flalign}\nonumber
		\begin{split}
		&\int_{\{-T_2\leq\psi<-t\}}|F_t|^2e^{-\varphi}\mathbb{I}_E(-\psi)\\
		\geq&\lim_{j\rightarrow +\infty}\int_{\{-T_2\leq\psi<-t\}}|F_t|^2e^{-\varphi}\mathbb{I}_{K_j}(-\psi)\\
		\geq&\lim_{j\rightarrow +\infty}\int_{\{-T_2\leq\psi<-t\}}|F|^2e^{-\varphi}\mathbb{I}_{K_j}(-\psi)\\
		=&\int_{\{-T_2\leq\psi<-t\}}|F|^2e^{-\varphi}\mathbb{I}_E(-\psi).
	\end{split}
	\end{flalign}
	For the measurable function $\tilde{c}(s)$, we can find an increasing sequence of simple functions $\{\sum_{j=1}^{n_i}a_{ij}\mathbb{I}_{E_{ij}}\}_{i=1}^{+\infty}$ on $(t,T_2]$,  such that $\lim\limits_{i\rightarrow+\infty}\sum_{j=1}^{n_i}a_{ij}\mathbb{I}_{E_{ij}}(s)=\tilde{c}(s)$ for a.e. $s\in (t,T_2]$. Then we can get that
	\begin{equation}\label{F_tgeqF}
		\int_{\{-T_2\leq\psi<-t\}}|F_t|^2e^{-\varphi}\tilde{c}(-\psi)\geq \int_{\{-T_2\leq\psi<-t\}}|F|^2e^{-\varphi}\tilde{c}(-\psi).
	\end{equation}
	Combining inequality (\ref{F_tgeqF}) with the assumption that
	\begin{equation}\nonumber
		\int_{\{\psi<-T_2\}}|F|^2e^{-\varphi}\tilde{c}(-\psi)=G(t;\tilde{c}),
	\end{equation}
	which means that
	\begin{equation}\nonumber
		\int_{\{\psi<-T_2\}}|F_t|^2e^{-\varphi}\tilde{c}(-\psi)\geq \int_{\{\psi<-T_2\}}|F|^2e^{-\varphi}\tilde{c}(-\psi),
	\end{equation}
	we get that
		\begin{equation}\nonumber
		G(t;\tilde{c})=\int_{\{\psi<-t\}}|F_t|^2e^{-\varphi}\tilde{c}(-\psi)\geq \int_{\{\psi<-t\}}|F|^2e^{-\varphi}\tilde{c}(-\psi).
	\end{equation}
	Then it follows from $(F-f)\in H^0(Z_0,(\mathcal{O}(K_M)\otimes\mathcal{F})|_{Z_0})$ that $F_t=F|_{\{\psi<-t\}}$ and $G(t;\tilde{c})=\int_{\{\psi<-t\}}|F|^2e^{-\varphi}\tilde{c}(-\psi)$. The other results are clear after these.
\end{proof}

We give the proof of Theorem \ref{s-con}.
\begin{proof}[Proof of Theorem \ref{s-con}]
	According to Remark \ref{tildec} and the statement (4) in Theorem \ref{s-con}, we get that $G(h^{-1}(r);\tilde{\varphi},\tilde{\psi},c)$ is linear respect to $r\in [\int_{T_2}^{+\infty}c(s)e^{-s}\mathrm{d}s$, $\int_{T_1}^{+\infty}c(s)e^{-s}\mathrm{d}s]$, and
	\begin{equation}\label{s-con2}
		G(t;,\tilde{\varphi},\tilde{\psi},c)=\int_{\{\tilde{\psi}<-t\}}|\tilde{F}|^2e^{-\tilde{\varphi}}c(-\tilde{\psi}), \ \forall t\in [T_1,T_2].
	\end{equation}
	By the statement (3) in Theorem \ref{s-con}, there is $(\tilde{F}-f)\in H^0(\tilde{Z}_0,(\mathcal{O}(K_M)\otimes\mathcal{I}(\tilde{\varphi}+\tilde{\psi}))|_{\tilde{Z}_0})$ $\subset H^0(Z_0,(\mathcal{O}(K_M)\otimes\mathcal{I}(\varphi+\psi))|_{Z_0})$ $\subset H^0(Z_0,(\mathcal{O}(K_M)\otimes\mathcal{F})|_{Z_0})$. Then for any $t\in [T_1,T_2]$, and any holomorphic $(n,0)$ form $F_t$ on $\{\psi<-t\}$ satisfying $(F_t-f)\in H^0(Z_0,(\mathcal{O}(K_M)\otimes\mathcal{F})|_{Z_0})$, on the one hand, equality (\ref{s-con1}) and Lemma \ref{F_t} shows that
	\begin{equation}\label{s-con3}
		\int_{\{\psi<-T_2\}}|F_t|^2e^{-\varphi}c(-\psi)\geq G(T_2;\varphi,\psi,c)=\int_{\{\psi<-T_2\}}|\tilde{F}|^2e^{-\varphi}c(-\psi).
	\end{equation}
	On the other hand, according to the statement (3), we have $(F_t-f)\in H^0(Z_0,(\mathcal{O}(K_M)\otimes\mathcal{F}))|_{Z_0})$ $\subset H^0(\tilde{Z}_0,(\mathcal{O}(K_M)\otimes\tilde{\mathcal{F}})|_{\tilde{Z}_0})$. Then it follows from the statement (1), (2), equality (\ref{s-con2}) and Lemma \ref{F_t} that
	\begin{flalign}\label{s-con4}
		\begin{split}
			&\int_{\{-T_2\leq\psi<-t\}}|F_t|^2e^{-\varphi}c(-\psi)\\
			=&\int_{\{-T_2\leq\tilde{\psi}<-t\}}|F_t|^2e^{-\tilde{\varphi}}c(-\tilde{\psi})\\
			=&\int_{\{\tilde{\psi}<-t\}}|F_t|^2e^{-\tilde{\varphi}}c(-\tilde{\psi})-\int_{\{\tilde{\psi}<-T_2\}}|F_t|^2e^{-\tilde{\varphi}}c(-\tilde{\psi})\\
			=&\int_{\{\tilde{\psi}<-t\}}|\tilde{F}|^2e^{-\tilde{\varphi}}c(-\tilde{\psi})+\int_{\{\tilde{\psi}<-t\}}|F_t-\tilde{F}|^2e^{-\tilde{\varphi}}c(-\tilde{\psi})\\
			&-\int_{\{\tilde{\psi}<-T_2\}}|\tilde{F}|^2e^{-\tilde{\varphi}}c(-\tilde{\psi})-\int_{\{\tilde{\psi}<-T_2\}}|F_t-\tilde{F}|^2e^{-\tilde{\varphi}}c(-\tilde{\psi})\\
			=&\int_{\{-T_2\leq\tilde{\psi}<-t\}}|\tilde{F}|^2e^{-\tilde{\varphi}}c(-\tilde{\psi})+\int_{\{-T_2\leq\tilde{\psi}<-t\}}|F_t-\tilde{F}|^2e^{-\tilde{\varphi}}c(-\tilde{\psi})\\
			\geq&\int_{\{-T_2\leq\tilde{\psi}<-t\}}|\tilde{F}|^2e^{-\tilde{\varphi}}c(-\tilde{\psi})\\
			=&\int_{\{-T_2\leq\psi<-t\}}|\tilde{F}|^2e^{-\varphi}c(-\psi).
		\end{split}
	\end{flalign}
	Combining inequality (\ref{s-con3}) with equality (\ref{s-con4}), we get that
	\begin{equation}\nonumber
		\int_{\{\psi<-t\}}|F_t|^2e^{-\varphi}c(-\psi)\geq\int_{\{\psi<-t\}}|\tilde{F}|^2e^{-\varphi}c(-\psi).
	\end{equation}
	Since $(\tilde{F}-f)\in H^0(Z_0,(\mathcal{O}(K_M)\otimes\mathcal{F})|_{Z_0})$, according to the arbitrariness of $F_t$, we know that
	\begin{equation}\nonumber
		G(t;\varphi,\psi,c)=\int_{\{\psi<-t\}}|\tilde{F}|^2e^{-\varphi}c(-\psi), \ t\in [T_1,T_2].
	\end{equation}
	Now for any $t\in [T_1,T_2]$, Remark \ref{tildec} implies that
	\begin{flalign}\nonumber
		\begin{split}
			G(t;\varphi,\psi,c)&=\int_{\{\psi<-t\}}|\tilde{F}|^2e^{-\varphi}c(-\psi)\\
			&=\int_{\{\psi<-T_2\}}|\tilde{F}|^2e^{-\varphi}c(-\psi)+\int_{\{-T_2\leq\psi<-t\}}|\tilde{F}|^2e^{-\varphi}c(-\psi)\\
			&=G(T_2;\varphi,\psi,c)+\int_{\{-T_2\leq\tilde{\psi}<-t\}}|\tilde{F}|^2e^{-\tilde{\varphi}}c(-\tilde{\psi})\\
			&=G(T_2;\varphi,\psi,c)+\frac{G(T_1;\tilde{\varphi},\tilde{\psi},\tilde{c})-G(T_2;\tilde{\varphi},\tilde{\psi},\tilde{c})}{\int_{T_1}^{T_2}\tilde{c}(s)e^{-s}\mathrm{d}s}\int_t^{T_2}c(s)e^{-s}\mathrm{d}s.
		\end{split}
	\end{flalign}
	It means that $G(h^{-1}(r);\varphi,\psi,c)$ is linear with respect to $r\in [\int_{T_2}^{+\infty}c(s)e^{-s}\mathrm{d}s$, $\int_{T_1}^{+\infty}c(s)e^{-s}\mathrm{d}s]$.
\end{proof}

\section{Proofs of Theorem \ref{sp1} and Theorem \ref{sp2}}

In this section, we prove Theorem \ref{sp1} and Theorem \ref{sp2}.

\begin{proof}[Proof of Theorem \ref{sp1}]
Assume that $G(h^{-1}(r);\varphi)$ is linear with respect to $r\in [\int_{T_2}^{+\infty}c(s)e^{-s}\mathrm{d}s$, $\int_{T_1}^{+\infty}c(s)e^{-s}\mathrm{d}s]$. Then it follows from Theorem \ref{partiallylinear} that there exists a unique holomorphic $(n,0)$ form $F$ on $\{\psi<-T_1\}$ satisfying $(F-f)\in H^0(Z_0,(\mathcal{O}(K_M)\otimes\mathcal{F})|_{Z_0})$, and 
\begin{equation}\nonumber
G(t;\varphi)=\int_{\{\psi<-t\}}|F|^2e^{-\varphi}c(-\psi)
\end{equation}
for any $t\in [T_1,T_2]$. 

As $\tilde{\varphi}+\psi$ is plurisubharmonic and $\tilde{\varphi}-\varphi$ is bounded on $\{\psi<-T_1\}$, it follows from Theorem \ref{Concave} that $G(h^{-1}(r);\tilde{\varphi})$ is concave with respect to $r\in [0,\int_{T_1}^{+\infty}c(s)e^{-s}\mathrm{d}s]$. Since $\tilde{\varphi}+\psi\geq \varphi+\psi$ and $\tilde{\varphi}+\psi\not\equiv\varphi+\psi$ on the interior  of $\{-T_2\leq\psi<-T_1\}$, and both of them are plurisubharmonic on $M$, then there exists a subset $U$ of $\{-T_2\leq\psi<-T_1\}$ such that $\mu(U)>0$ and $e^{-\tilde{\varphi}}<e^{-\varphi}$ on $U$, where $\mu$ is the Lebesgue measure on $M$. Since $\tilde{\varphi}=\varphi$ on $\{\psi<-T_2\}$, we have
\begin{flalign}\nonumber
	\begin{split}
	&\frac{G(T_1;\tilde{\varphi})-G(T_2;\tilde{\varphi})}{\int_{T_1}^{T_2}c(s)e^{-s}\mathrm{d}s}\\
	=&\frac{G(T_1;\tilde{\varphi})-G(T_2;\varphi)}{\int_{T_1}^{T_2}c(s)e^{-s}\mathrm{d}s}\\
	\leq&\frac{\int_{\{\psi<-T_1\}}|F|^2e^{-\tilde{\varphi}}c(-\psi)-\int_{\{\psi<-T_2\}}|F|^2e^{-\varphi}c(-\psi)}{\int_{T_1}^{T_2}c(s)e^{-s}\mathrm{d}s}\\
	<&\frac{G(T_1;\varphi)-G(T_2;\varphi)}{\int_{T_1}^{T_2}c(s)e^{-s}\mathrm{d}s}
	\end{split}
\end{flalign}
for any $t\in [T_1,T_2)$. By Lemma \ref{F_t}, for any $t\in [T_1,T_2)$, there exists a holomorphic $(n,0)$ form $F_t$ such that $(F_t-f)\in H^0(Z_0,(\mathcal{O}(K_M)\otimes\mathcal{F})|_{Z_0})$, and 
\begin{equation}\nonumber
G(t;\tilde{\varphi})=\int_{\{\psi<-t\}}|F_t|^2e^{-\tilde{\varphi}}c(-\psi)<+\infty.
\end{equation}
Since $\tilde{\varphi}-\varphi$ is bounded on $\{\psi<-T_1\}$, we have
\begin{equation}\nonumber
	\int_{\{\psi<-t\}}|F_t|^2e^{-\varphi}c(-\psi)<+\infty.
\end{equation}
Then according to Lemma \ref{F_t}, for any $t_1,t_2\in [T_1,T_2]$, $t_1<t_2$, we have
\begin{flalign}\nonumber
	\begin{split}
		&G(t_1;\tilde{\varphi})-G(t_2;\tilde{\varphi})\\
		\geq&\int_{\{-t_2\leq\psi<-t_1\}}|F_{t_1}|^2e^{-\tilde{\varphi}}c(-\psi)\\
		\geq&(\inf_{\{-t_2\leq\psi<-t_1\}}e^{\varphi-\tilde{\varphi}})\int_{\{-t_2\leq\psi<-t_1\}}|F_{t_1}|^2e^{-\varphi}c(-\psi)\\
		=&(\inf_{\{-t_2\leq\psi<-t_1\}}e^{\varphi-\tilde{\varphi}})\times\\
		&\left(\int_{\{-t_2\leq\psi<-t_1\}}|F|^2e^{-\varphi}c(-\psi)+\int_{\{-t_2\leq\psi<-t_1\}}|F_{t_1}-F|^2e^{-\varphi}c(-\psi)\right)\\
		\geq&(\inf_{\{-t_2\leq\psi<-t_1\}}e^{\varphi-\tilde{\varphi}})\int_{\{-t_2\leq\psi<-t_1\}}|F|^2e^{-\varphi}c(-\psi).
	\end{split}
\end{flalign}
Then $\lim\limits_{t\rightarrow T_1+0}\sup\limits_{\{-t\leq \psi<-T_1\}}(\tilde{\varphi}-\varphi)=0$ implies that
\begin{flalign}\nonumber
	\begin{split}
		&\lim_{t\rightarrow T_1+0}\frac{G(T_1;\tilde{\varphi})-G(t;\tilde{\varphi})}{\int_{T_1}^tc(s)e^{-s}\mathrm{d}s}\\
		\geq&\lim_{t\rightarrow T_1+0}(\inf_{\{-t\leq\psi<-T_1\}}e^{\varphi-\tilde{\varphi}})\frac{\int_{\{-t\leq\psi<-T_1\}}|F|^2e^{-\varphi}c(-\psi)}{\int_{T_1}^tc(s)e^{-s}\mathrm{d}s}\\
		=&\frac{G(T_1;\varphi)-G(T_2;\varphi)}{\int_{T_1}^{T_2}c(s)e^{-s}\mathrm{d}s}\\
		>&\frac{G(T_1;\tilde{\varphi})-G(T_2;\tilde{\varphi})}{\int_{T_1}^{T_2}c(s)e^{-s}\mathrm{d}s},
	\end{split}
\end{flalign}
which contradicts to the concavity of $G(h^{-1}(r);\tilde{\varphi})$. It means that the assumption can not hold, i.e. $G(h^{-1}(r);\varphi)$ is  not linear with respect to $r$ on $ [\int_{T_2}^{+\infty}c(s)e^{-s}\mathrm{d}s$, $\int_{T_1}^{+\infty}c(s)e^{-s}\mathrm{d}s]$.

Especially, if $\varphi+\psi$ is strictly plurisubharmonic at $z_1\in \text{int}(\{-T_2\leq\psi<-T_1\})$, we can construct some $\tilde{\varphi}$ satisfying the five statements in Theorem \ref{sp1}. By the assumption, there is a small open neighborhood $(U, w)$ of $z_1$, such that $U\Subset \{-T_2\leq\psi<-T_1\}$ and $\sqrt{-1}\partial\bar{\partial}(\varphi+\psi)>\varepsilon \omega$ on $U$, where $w=(w_1,\ldots,w_n)$ is the local coordinate on $U$, $\omega=\sqrt{-1}\sum_{j=1}^n\mathrm{d}w_j\wedge \mathrm{d}\bar{w}_j$ on $U$. Let $\rho$ be a smooth nonnegative function on $M$ such that $\rho\not\equiv 0$ and $\text{Supp} \rho\Subset U$. Then we can choose a positive number $\delta$ such that
\begin{equation}\nonumber
	\sqrt{-1}\partial\bar{\partial}(\varphi+\psi+\delta\rho)\geq 0
\end{equation}
on $U$. Let $\tilde{\varphi}=\varphi+\delta\rho$, then it can be checked that $\tilde{\varphi}$ satisfies the five statements in Theorem \ref{sp1}. It implies that $G(h^{-1}(r))$ is not linear with respect to $r$ on $ [\int_{T_2}^{+\infty}c(s)e^{-s}\mathrm{d}s$, $\int_{T_1}^{+\infty}c(s)e^{-s}\mathrm{d}s]$.
\end{proof}

Now, we give the proof of Theorem \ref{sp2}.

\begin{proof}[Proof of Theorem \ref{sp2}]
	Let $\tilde{\varphi}=\varphi+\psi-\tilde{\psi}$, then $\tilde{\varphi}+\tilde{\psi}=\varphi+\psi$ is a plurisubharmonic function on $M$. Assume that $G(h^{-1}(r))$ is linear with respect to $r\in [\int_{T_2}^{+\infty}c(s)e^{-s}\mathrm{d}s,\int_{T_1}^{+\infty}c(s)e^{-s}\mathrm{d}s]$.
	
	 We can denote
	\begin{equation}\nonumber
				\tilde{c}(t)=\left\{
		\begin{array}{ll}
			c(T_1), & t\in[T_1,T_2], \\
			c(t), & t\in [T,+\infty]\setminus [T_1,T_2].
		\end{array}
		\right.
	\end{equation}	
	Then it is clear that $G(h^{-1}(r);\tilde{c})$ is also linear with respect to $r$ on $ [\int_{T_2}^{+\infty}c(s)e^{-s}\mathrm{d}s$, $\int_{T_1}^{+\infty}c(s)e^{-s}\mathrm{d}s]$ by Remark \ref{tildec}. Thus we may also assume that $c(t)e^{-t}$ is strictly decreasing on $[T_1,T_2]$ (note that $c(T_1)\leq c(T_2)e^{T_1-T_2}<c(T_2)$), and $c(t)$ is increasing on $[T_1,T_2]$.
	
	It follows from Theorem \ref{partiallylinear} that there exists a unique holomorphic $(n,0)$ form $F$ on $\{\psi<-T_1\}$ satisfying $(F-f)\in H^0(Z_0,(\mathcal{O}(K_M)\otimes\mathcal{F})|_{Z_0})$, and 
	\begin{equation}\nonumber
		G(t;\varphi,\psi)=\int_{\{\psi<-t\}}|F|^2e^{-\varphi}c(-\psi)
	\end{equation}
	for any $t\in [T_1,T_2]$. 
	
	According to the assumptions, we have $Z_0\subset\{\psi=-\infty\}=\{\tilde{\psi}=-\infty\}$. As $c(t)e^{-t}$ is decreasing and $\tilde{\psi}\geq\psi$, we have $e^{-\varphi}c(-\psi)=e^{-\varphi-\psi}e^{\psi}c(-\psi)\leq e^{-\tilde{\varphi}-\tilde{\psi}}e^{\tilde{\psi}}c(-\tilde{\psi})=e^{-\tilde{\varphi}}c(-\tilde{\psi})$. Then it follows from Theorem \ref{Concave} that $G(h^{-1}(r);\tilde{\varphi},\tilde{\psi})$ is concave with respect to $r$. We prove the following inequality:
	\begin{equation}\label{sp2-1}
		\lim_{t\rightarrow T_1+0}\frac{G(t;\tilde{\varphi},\tilde{\psi})-G(T_2;\tilde{\varphi},\tilde{\psi})}{\int_t^{T_2}c(s)e^{-s}\mathrm{d}s}>\frac{G(T_1;\varphi,\psi)-G(T_2;\varphi,\psi)}{\int_{T_1}^{T_2}c(s)e^{-s}\mathrm{d}s}.
	\end{equation}
	We just need to prove it for the case $G(T_1;\tilde{\varphi},\tilde{\psi})<+\infty$. By Lemma \ref{F_t}, there exists a holomorphic $(n,0)$ form $F_{T_1}$ such that $(F_{T_1}-f)\in H^0(Z_0,(\mathcal{O}(K_M)\otimes\mathcal{F})|_{Z_0})$, and 
	\begin{equation}\nonumber
		G(T_1;\tilde{\varphi},\tilde{\psi})=\int_{\{\psi<-T_1\}}|F_{T_1}|^2e^{-\tilde{\varphi}}c(-\tilde{\psi})\in (0,+\infty).
	\end{equation}	
	Since $\tilde{\psi}\geq\psi$ and $\tilde{\psi}\not\equiv\psi$ on the interior  of $\{-T_2\leq\psi<-T_1\}$, and both of them are plurisubharmonic on $M$, then there exists a subset $U$ of $\{-T_2\leq\psi<-T_1\}$ such that $\mu(U)>0$ and $e^{-\tilde{\psi}}<e^{-\psi}$ on $U$, where $\mu$ is the Lebesgue measure on $M$. As $F_{T_1}\not\equiv 0$,  $\tilde{\psi}=\psi$ on $\{\psi<-T_2\}$, and $c(t)e^{-t}$ is strictly decreasing on $[T_1,T_2]$, we have
	\begin{flalign}\nonumber
		\begin{split}
			&\frac{G(T_1;\tilde{\varphi},\tilde{\psi})-G(T_2;\tilde{\varphi},\tilde{\psi})}{\int_{T_1}^{T_2}c(s)e^{-s}\mathrm{d}s}\\
			=&\frac{\int_{\{\psi<-T_1\}}|F_{T_1}e^{-\tilde{\varphi}}c(-\tilde{\psi})-G(T_2;\varphi,\psi)}{\int_{T_1}^{T_2}c(s)e^{-s}\mathrm{d}s}\\
			>&\frac{\int_{\{\psi<-T_1\}}|F_{T_1}e^{-\varphi}c(-\psi)-G(T_2;\varphi,\psi)}{\int_{T_1}^{T_2}c(s)e^{-s}\mathrm{d}s}\\
			\geq&\frac{G(T_1;\varphi,\psi)-G(T_2;\varphi,\psi)}{\int_{T_1}^{T_2}c(s)e^{-s}\mathrm{d}s}.
		\end{split}
	\end{flalign}
	Thus we get inequality (\ref{sp2-1}).
	
	As $c(t)$ is increasing on $[T_1,T_2]$ and $\lim\limits_{t\rightarrow T_2-0}\sup\limits_{\{-T_2\leq \psi<-t\}}(\tilde{\psi}-\psi)=0$, we obtain that
	\begin{flalign}\label{sp2-2}
		\begin{split}
			&\lim_{t\rightarrow T_2-0}\frac{G(t;\tilde{\varphi},\tilde{\psi})-G(T_2;\tilde{\varphi},\tilde{\psi})}{\int_t^{T_2}c(s)e^{-s}\mathrm{d}s}\\
			=&\lim_{t\rightarrow T_2-0}\frac{G(t;\tilde{\varphi},\tilde{\psi})-G(T_2;\varphi,\psi)}{\int_t^{T_2}c(s)e^{-s}\mathrm{d}s}\\
			\leq&\lim_{t\rightarrow T_2-0}\frac{\int_{\{-T_2\leq\psi<-t\}}|F|^2e^{-\tilde{\varphi}}c(-\psi)}{\int_t^{T_2}c(s)e^{-s}\mathrm{d}s}\\
			\leq&\lim_{t\rightarrow T_2-0}\frac{\int_{\{-T_2\leq\psi<-t\}}|F|^2e^{-\varphi-\psi}e^{\tilde{\psi}}c(-\psi)}{\int_t^{T_2}c(s)e^{-s}\mathrm{d}s}\\
			\leq&\lim_{t\rightarrow T_2-0}(\sup_{\{-T_2\leq\psi<-t\}}e^{\tilde{\psi}-\psi})\frac{\int_{\{-T_2\leq\psi<-t\}}|F|^2e^{-\varphi}c(-\psi)}{\int_t^{T_2}c(s)e^{-s}\mathrm{d}s}\\
			=&\frac{\int_{\{-T_2\leq\psi<-T_1\}}|F|^2e^{-\varphi}c(-\psi)}{\int_{T_1}^{T_2}c(s)e^{-s}\mathrm{d}s}\\
			=&\frac{G(T_1;\varphi,\psi)-G(T_2;\varphi,\psi)}{\int_{T_1}^{T_2}c(s)e^{-s}\mathrm{d}s}.
		\end{split}
	\end{flalign}
	Combining (\ref{sp2-2}) with (\ref{sp2-1}), we have
	\begin{equation}\nonumber
		\lim_{t\rightarrow T_2-0}\frac{G(t;\tilde{\varphi},\tilde{\psi})-G(T_2;\tilde{\varphi},\tilde{\psi})}{\int_t^{T_2}c(s)e^{-s}\mathrm{d}s}<\lim_{t\rightarrow T_1+0}\frac{G(t;\tilde{\varphi},\tilde{\psi})-G(T_2;\tilde{\varphi},\tilde{\psi})}{\int_t^{T_2}c(s)e^{-s}\mathrm{d}s},
	\end{equation}
	which contradicts to the concavity of $G(h^{-1}(r);\tilde{\varphi},\tilde{\psi})$. It means that the assumptions can not hold, i.e. $G(h^{-1}(r);\varphi,\psi)$ is  not linear with respect to $r$ on $ [\int_{T_2}^{+\infty}c(s)e^{-s}\mathrm{d}s$, $\int_{T_1}^{+\infty}c(s)e^{-s}\mathrm{d}s]$.
	
	Especially, if $\psi$ is strictly plurisubharmonic at $z_1\in \text{int}(\{-T_2\leq\psi<-T_1\})$, we can construct some $\tilde{\psi}$ satisfying the four statements in Theorem \ref{sp2}. By the assumption, there is a small open neighborhood $(U, w)$ of $z_1$, such that $U\Subset \{-T_2\leq\psi<-T_1\}$ and $\sqrt{-1}\partial\bar{\partial}(\varphi+\psi)>\varepsilon \omega$ on $U$, where $w=(w_1,\ldots,w_n)$ is the local coordinate on $U$, $\omega=\sqrt{-1}\sum_{j=1}^n\mathrm{d}w_j\wedge \mathrm{d}\bar{w}_j$ on $U$. Let $\rho$ be a smooth nonnegative function on $M$ such that $\rho\not\equiv 0$ and $\text{Supp} \rho\Subset U$. Then we can choose a positive number $\delta$ such that
	\begin{equation}\nonumber
		\sqrt{-1}\partial\bar{\partial}(\varphi+\psi+\delta\rho)\geq 0
	\end{equation}
	 and $\psi+\delta\rho<-T_1$ on $U$. Let $\tilde{\psi}=\psi+\delta\rho$, then it can be checked that $\tilde{\psi}$ satisfies the four statements in Theorem \ref{sp2}. It implies that $G(h^{-1}(r))$ is not linear with respect to $r$ on $ [\int_{T_2}^{+\infty}c(s)e^{-s}\mathrm{d}s$, $\int_{T_1}^{+\infty}c(s)e^{-s}\mathrm{d}s]$.
\end{proof}

\section{Proof of Theorem \ref{thm:char}}\label{sec:proof of theorem char}

In this section, we prove Theorem \ref{thm:char}.

\begin{proof}
	[Proof of Theorem \ref{thm:char}]
	We proof Theorem \ref{thm:char} in two steps: Firstly, we prove the  sufficiency part of characterization; Secondly, we prove the necessity part of characterization.

	\emph{Step 1.}  Assume that the two statements in Theorem \ref{thm:char} hold. 
	
 Note that $\psi=G_{\Omega}(\cdot,z_0)+a$ and $\varphi=2\log|g|+2G_{\Omega}(\cdot,z_0)+2u$ on $\{\psi<2a\}=\{G_{\Omega}(\cdot,z_0)<a\}$. It follows from $\chi_{z_0}=\chi_{-u}$ that there is a holomorphic function $g_0$ on $\Omega$ such that $|g_0|=e^{u+G_{\Omega}(\cdot,z_0)}$, thus $$\chi_{a,z_0}=\chi_{a,-u},$$ where $\chi_{a,z_0}$ and $\chi_{a,-u}$ are the characters on $\{G_{\Omega}(\cdot,z_0)<a\}$ associated to $G_{\Omega}(\cdot,z_0)$ and $-u$ respectively. Denote that $\psi_1:=\psi-2a=G_{\Omega}(\cdot,z_0)-a$ (the Green function on $\{G_{\Omega}(\cdot,z_0)<a\}$) and $c_1(t):=c(t-2a)$. By Theorem \ref{thm:e2}, $G(h_1^{-1}(r);\psi_1,c_1)$ is linear on $[0,\int_0^{+\infty}c_1(t)e^{-t}dt]=[0,e^{-2a}\int_{-2a}^{+\infty}c(t)e^{-t}dt]$, where $h_1(t)=\int_t^{+\infty}c(se^{-s}ds)$. Note that 
 $$G(h_1^{-1}(r);\psi_1,c_1)=G(h^{-1}(re^{2a});\psi,c),$$
then we get that $G(h^{-1}(r);\psi,c)$ is linear on $[0,\int_{-2a}^{+\infty}c(t)e^{-t}dt]$.
	
Denote that $\psi_2:=2G_{\Omega}(\cdot,z_0)$ and $\varphi_2:=2\log|g|+2u+a$ on $\Omega$. By Theorem \ref{thm:e2}, $G(h^{-1}(r);\psi_2,\varphi_2)$ is linear on $[0,\int_0^{+\infty}c(t)e^{-t}dt]$. Denote that
$$F:=b_0gp_*(f_{u}df_{z_0}),$$
where $b_0$ is a constant such that $ord_{z_0}(F-f_0)>k$.
 Following from Remark \ref{r:e2}, we have
 \begin{equation}
 	\label{eq:230109a}
 	\begin{split}
 		G(t-2a;\psi,\varphi,c)&=G(t;\psi_1,\varphi,c_1)\\
 		&=\int_{\{\psi_1<-t\}}|F|^2e^{-\varphi}c_1(-\psi_1)\\
 		&=\int_{\{\psi<-t+2a\}}|F|^2e^{-\varphi}c(-\psi)
 	\end{split}
 \end{equation}
and 
\begin{equation}
\label{eq:230109b}
G(t;\psi_2,\varphi_2,c)=\int_{\{\psi_2<-t\}}|F|^2e^{-\varphi_2}c(-\psi_2)
\end{equation}
for any $t\ge0$. 

Let $t_1\in[0,-2a]$, and let $\tilde F$ be any holomorphic $(1,0)$ form on $\{\psi<-t_1\}$ satisfying $(\tilde F-f_0,z_0)\in(\mathcal{I}(\varphi+\psi)\otimes\mathcal{O}(K_{\Omega}))_{z_0}$ and $\int_{\{\psi<-t_1\}}|\tilde F|^2e^{-\varphi}c(-\psi)<+\infty$. As $c(t)e^{-t}$ is decreasing and $\psi_2\le\psi$, then $e^{-\varphi_2}c(-\psi_2)\le e^{-\varphi}c(-\psi)$. By Lemma \ref{F_t} and equality \eqref{eq:230109b}, we have 
\begin{displaymath}
	\begin{split}
		&\int_{\{\psi<-t_1\}}|\tilde F|^2e^{-\varphi_2}c(-\psi_2)\\
		=&\int_{\{\psi_2<-t_1\}}|\tilde F|^2e^{-\varphi_2}c(-\psi_2)\\
		=&\int_{\{\psi_2<-t_1\}}| F|^2e^{-\varphi_2}c(-\psi_2)+\int_{\{\psi_2<-t_1\}}|F-\tilde F|^2e^{-\varphi_2}c(-\psi_2)
	\end{split}
\end{displaymath}
and 
\begin{displaymath}
	\begin{split}
		&\int_{\{\psi<2a\}}|\tilde F|^2e^{-\varphi_2}c(-\psi_2)\\
		=&\int_{\{\psi_2<2a\}}|\tilde F|^2e^{-\varphi_2}c(-\psi_2)\\
		=&\int_{\{\psi_2<2a\}}| F|^2e^{-\varphi_2}c(-\psi_2)+\int_{\{\psi_2<2a\}}|F-\tilde F|^2e^{-\varphi_2}c(-\psi_2),
	\end{split}
\end{displaymath}
which shows that
\begin{equation}
	\nonumber
	\begin{split}
			\int_{\{2a\le\psi<-t_1\}}|\tilde F|e^{-\varphi}c(-\psi)&=\int_{\{2a\le\psi<-t_1\}}|\tilde F|e^{-\varphi}c(-\psi)\\
			&\ge\int_{\{2a\le\psi<-t_1\}}|F|^2e^{-\varphi}c(-\psi).
	\end{split}
\end{equation}
Equality \eqref{eq:230109a} implies that 
$$	\int_{\{\psi<2a\}}|\tilde F|e^{-\varphi}c(-\psi)\ge\int_{\{\psi<2a\}}|F|^2e^{-\varphi}c(-\psi).$$ 
Thus, we have 
\begin{equation}
	\nonumber
	\begin{split}
		G(t_1;\psi,\varphi,c)&=\int_{\{\psi<-t_1\}}|F|^2e^{-\varphi}c(-\psi)\\
		&=\int_{\{\psi<2a\}}|F|^2e^{-\varphi}c(-\psi)+\int_{\{2a\le\psi<-t_1\}}|F|^2e^{-\varphi}c(-\psi)\\
		&=\int_{\{\psi<2a\}}|F|^2e^{-\varphi}c(-\psi)+\int_{\{2a\le\psi_2<-t_1\}}|F|^2e^{-\varphi_2}c(-\psi_2)\\
		&=\int_{\{\psi<2a\}}|F|^2e^{-\varphi}c(-\psi)+G(t_1;\psi_2,\varphi_2,c)-G(-2a;\psi_2,\varphi_2,c)
	\end{split}
\end{equation}
for any $t_1\in[0,-2a]$. As $G(h^{-1}(r);\psi_2,\varphi_2,c)$ is linear on $[0,\int_0^{+\infty}c(t)e^{-t}dt]$, then $G(h^{-1}(r);\psi,\varphi,c)$ is linear on $[\int_{-2a}^{+\infty}c(t)e^{-t}dt,\int_0^{+\infty}c(t)e^{-t}dt]$.

\

\emph{Step 2.} Assume that $G(h^{-1}(r))$ is linear with respect to $r$ on $[0,\int_{-2a}^{+\infty}c(t)e^{-t}dt]$ and $[\int_{-2a}^{+\infty}c(t)e^{-t}dt,\int_0^{+\infty}c(t)e^{-t}dt]$.

By Theorem \ref{partiallylinear}, there exists a holomorphic $(1,0)$ form $F_1$ on $\Omega$ such that $(F_1-f_0,z_0)\in(\mathcal{I}(\varphi+\psi)\otimes\mathcal{O}(K_{\Omega}))_{z_0}$ and 
\begin{equation}
	\label{eq:0111a}
	G(t)=\int_{\{\psi<-t\}}|F_1|^2e^{-\varphi}c(-\psi)
\end{equation}
for any $t\ge0$.
Denote that 
$$h:=\frac{F_1}{p_*(df_{z_0})}$$
is a multi-valued meromorphic function on $\Omega$, and $|h|$ is single-valued. As $G(h^{-1}(r))$ is linear on $[0,\int_{-2a}^{+\infty}c(t)e^{-t}dt]$, it follows from Remark \ref{r:e2} that 
$$2\log|h|=\varphi_0+b_1$$
on $\{G_{\Omega}(\cdot,z_0)<a\}$, where $b_1$ is a constant. As $\varphi_0$ is subharmonic on $\Omega$, $h$ has no pole in $\{G_{\Omega}(\cdot,z_0)\le a\}$. Since $\{G_{\Omega}(\cdot,z_0)=a\}$ is a compact set, by using Lemma \ref{l:extra}, we know that there exists $a_1\in(a,0)$ such that $h$ has no zero point in $\{a<G_{\Omega}(\cdot,z_0)<a_1\}$. It follows from the Weierstrass Theorem on open Riemann surfaces (see \cite{OF81}), that there is a holomorphic function $g_1$ on $\Omega$ such that
 $$u_1:=\log|h|-\log|g_1|-\frac{b_1}{2}$$
 is harmonic on $\{G_{\Omega}(\cdot,z_0)<a_1\}$. Thus, $g_1$ has no zero point in $\{a<G_{\Omega}(\cdot,z_0)<a_1\}$, and 
 $$u_1=\frac{\varphi_0}{2}-\log|g_1|$$
 on $\{G_{\Omega}(\cdot,z_0)<a\}$. 
  Note that $2\log|h|=\varphi_0+b_1$ on $\{G_{\Omega}(\cdot,z_0)<a\}$ and $\{G_{\Omega}(\cdot,z_0)=a\}$ is a closed real analytic curve, it follows from Lemma \ref{l:zero point} that the Lelong number $v(dd^c\varphi_0,z)\ge2ord_{z}(h)=2ord_z(g_1)$ for any $z\in\{G_{\Omega}(\cdot,z_0)=a\}$, hence 
$$v_1:=\frac{\varphi_0}{2}-\log|g_1|$$ 
is subharmonic on $\{G_{\Omega}(\cdot,z_0)<a_1\}$. Note that
$v_1=u_1$
on $\{G_{\Omega}(\cdot,z_0)<a\}$, then 
$$v_1=u_1$$
on $\{G_{\Omega}(\cdot,z_0)\le a\}$ by Lemma \ref{l:harmonic subharmonic}.

As $G(h^{-1}(r))$ is linear on $[0,\int_{-2a}^{+\infty}c(t)e^{-t}dt]$, it follows from Theorem \ref{thm:e2} that 
$$ord_{z_0}(g_1)=ord_{z_0}(f_0)=ord_{z_0}(F_1).$$   Denote that $\psi_3:=2G_{\Omega}(\cdot,z_0)$ and $\varphi_3:=2\log|g_1|+2u_1+a$ on $\{G_{\Omega}(\cdot,z_0)<a_1\}$. Since $\frac{h}{g_1}dp_*(f_{z_0})$ is a single-value holomorphic $(1,0)$ form and $u_1=\log|\frac{h}{g_1}|-\frac{b_1}{2}$ on $\{G_{\Omega}(\cdot,z_0)<a_1\}$, we know that 
$$\chi_{a_1,-u_1}=\chi_{a_1,z_0}.$$ 
Following from Theorem \ref{thm:e2} and Remark \ref{r:e2}, $G(h^{-1}(r);\psi_3,\varphi_3)$ is linear  with respect to $r$ on $[0,\int_{-2a_1}^{+\infty}c(t)e^{-t}dt]$ and 
\begin{equation}
	\label{eq:0110a}
	G(t;\psi_3,\varphi_3)=\int_{\{\psi_3<-t\}}|F_1|^2e^{-\varphi_3}c(-\psi_3)
\end{equation}
for any $t\ge-2a_1$. Denote that $\tilde\varphi_3:=\varphi_0+a$ on $\Omega$, hence $\tilde\varphi_3=\varphi_3$ on $\{G_{\Omega}(\cdot,z_0)<a\}$ and $\tilde\varphi_3=2\log|g_1|+2v_1+a$ on $\{G_{\Omega}(\cdot,z_0)<a_1\}$. Theorem \ref{Concave} shows that $G(h^{-1}(r);\psi_3,\tilde\varphi_3)$ is concave on $[0,\int_0^{+\infty}c(t)e^{-t}dt]$. 

Let $t_1\in[0,-2a]$, and let $\tilde F$ be any holomorphic $(1,0)$ form on $\{\psi_3<-t_1\}=\{\psi<-t_1\}$ satisfying $ord_{z_0}(\tilde F-f_0)>ord_{z_0}(g_1)$ and 
$$\int_{\{\psi<-t_1\}}|\tilde F|^2e^{-\tilde\varphi_3}c(-\psi_3)<+\infty.$$
As $ord_{z_0}(\tilde F-F_1)>ord_{z_0}(g_1)$ and $c(t)e^{-t}$ is decreasing, it follows from Lemma \ref{l:G-compact} that there exists $t_2>t_1$ such that 
$$\int_{\{\psi_3<-t_2\}}|F_1-\tilde F|^2e^{-\varphi}c(-\psi)\le c(t_2)e^{-t_2}\int_{\{\psi_3<-t_2\}}|F_1-\tilde F|^2e^{-\varphi-\psi}<+\infty.$$
Note that $e^{-\varphi}c(-\psi)\le C_1e^{-\varphi_0+a}c(-2G_{\Omega}(\cdot,z_0))=C_1e^{-\tilde\varphi_3}c(-\psi_3)$ on $\{-t_2\le\psi_3<-t_1\}=\{-t_2\le 2G_{\Omega}(\cdot,z_0)<-t_1\}$, where $C_1>0$ is a constant. Then we have
\begin{equation}
	\nonumber
	\begin{split}
		&\int_{\{\psi<-t_1\}}|\tilde F|^2e^{-\varphi}c(-\psi)\\
		\le&\int_{\{\psi_3<-t_2\}}|\tilde F|^2e^{-\varphi}c(-\psi)+\int_{\{-t_2\le\psi_3<-t_1\}}|\tilde F|^2e^{-\varphi}c(-\psi)\\
		\le&2\int_{\{\psi_3<-t_2\}}|F_1-\tilde F|^2e^{-\varphi}c(-\psi)+\int_{\{\psi_3<-t_2\}}|F_1|^2e^{-\varphi}c(-\psi)\\
		&+C_1\int_{\{\psi_3<-t_2\}}|\tilde F|^2e^{-\tilde\varphi_3}c(-\psi_3)\\
		<&+\infty.
	\end{split}
\end{equation}
By Lemma \ref{F_t} and equality \eqref{eq:0111a}, we have 
\begin{displaymath}
	\begin{split}
		&\int_{\{\psi<-t_1\}}|\tilde F|^2e^{-\varphi}c(-\psi)\\
		=&\int_{\{\psi<-t_1\}}| F_1|^2e^{-\varphi}c(-\psi)+\int_{\{\psi<-t_1\}}|F_1-\tilde F|^2e^{-\varphi}c(-\psi)
	\end{split}
\end{displaymath}
and 
\begin{displaymath}
	\begin{split}
		&\int_{\{\psi<2a\}}|\tilde F|^2e^{-\varphi}c(-\psi)\\
		=&\int_{\{\psi<2a\}}| F_1|^2e^{-\varphi}c(-\psi)+\int_{\{\psi<2a\}}|F_1-\tilde F|^2e^{-\varphi}c(-\psi),
	\end{split}
\end{displaymath}
which shows that
\begin{equation}
	\nonumber
		\int_{\{2a\le\psi<-t_1\}}|\tilde F|e^{-\varphi}c(-\psi)\ge\int_{\{2a\le\psi<-t_1\}}|F_1|^2e^{-\varphi}c(-\psi).
\end{equation}
Combining equality \eqref{eq:0110a}, we have
\begin{equation}
	\nonumber
	\begin{split}
		&\int_{\{\psi_3<-t_1\}}|\tilde F|^2e^{-\tilde\varphi_3}c(-\psi_3)\\
		=&\int_{\{\psi_3<2a\}}|\tilde F|^2e^{-\varphi_3}c(-\psi_3)+\int_{\{2a\le\psi<-t_1\}}|\tilde F|^2e^{-\varphi}c(-\psi)\\
		\ge&\int_{\{\psi_3<2a\}}| F_1|^2e^{-\varphi_3}c(-\psi_3)+\int_{\{2a\le\psi<-t_1\}}| F_1|^2e^{-\varphi}c(-\psi)\\
		=&\int_{\{\psi_3<-t_1\}}|F_1|^2e^{-\tilde\varphi_3}c(-\psi_3)
	\end{split}
\end{equation}
for any $t_1\in[0,-2a]$.
Note that $G(t;\psi_3,\varphi_3)=	G(t;\psi_3,\tilde\varphi_3)$ for any $t\ge -2a$, then
\begin{equation}
	\label{eq:0111b}G(t;\psi_3,\tilde\varphi_3)=\int_{\{\psi_3<-t_1\}}|F_1|^2e^{-\tilde\varphi_3}c(-\psi_3)
\end{equation}
for any $t\ge0$. As $G(h^{-1}(r);\psi_3,\varphi_3)$ is linear on $[0,\int_{-2a}^{+\infty}c(t)e^{-t}dt]$ and $G(h^{-1}(r);\psi,\varphi)$ is linear on $[\int_{-2a}^{+\infty}c(t)e^{-t}dt,\int_0^{+\infty}c(t)e^{-t}dt]$, we get that $G(h^{-1}(r);\psi_3,\tilde\varphi_3)$ is linear on $[0,\int_{-2a}^{+\infty}c(t)e^{-t}dt]$ and $[\int_{-2a}^{+\infty}c(t)e^{-t}dt,\int_0^{+\infty}c(t)e^{-t}dt]$.

As  $G(h^{-1}(r);\psi_3,\varphi_3)$ is linear on $[0,\int_{-2a_1}^{+\infty}c(t)e^{-t}dt]$, we have
$$\lim_{s\rightarrow a+0}\frac{G(-2s;\psi_3,\varphi_3)-G(-2a;\psi_3,\varphi_3)}{\int_{-2s}^{+\infty}c(t)e^{-t}dt-\int_{-2a}^{+\infty}c(t)e^{-t}dt}=\lim_{s_1\rightarrow a-0}\frac{G(-2s_1;\psi_3,\varphi_3)-G(-2a;\psi_3,\varphi_3)}{\int_{-2s_1}^{+\infty}c(t)e^{-t}dt-\int_{-2a}^{+\infty}c(t)e^{-t}dt},$$
which shows that
\begin{equation}
	\label{eq:230111c}
	\begin{split}
			&\lim_{s\rightarrow a+0}\frac{\int_{\{2a\le\psi_3<2s\}}|p_*(df_{z_0})|^2e^{b_1-a}c(-\psi_3)}{\int_{-2s}^{-2a}c(t)e^{-t}dt}\\
			=&\lim_{s\rightarrow a+0}\frac{\int_{\{2a\le\psi_3<2s\}}|F_1|^2e^{-2u_1-2\log|g_1|-a}c(-\psi_3)}{\int_{-2s}^{-2a}c(t)e^{-t}dt}\\
			=&\lim_{s\rightarrow a+0}\frac{G(-2s;\psi_3,\varphi_3)-G(-2a;\psi_3,\varphi_3)}{\int_{-2s}^{+\infty}c(t)e^{-t}dt-\int_{-2a}^{+\infty}c(t)e^{-t}dt}\\
			=&\lim_{s_1\rightarrow a-0}\frac{G(-2s_1;\psi_3,\varphi_3)-G(-2a;\psi_3,\varphi_3)}{\int_{-2s_1}^{+\infty}c(t)e^{-t}dt-\int_{-2a}^{+\infty}c(t)e^{-t}dt}\\
			=&\lim_{s_1\rightarrow a-0}\frac{\int_{\{2s_1\le\psi_3<2a\}}|F_1|^2e^{-2u_1-2\log|g_1|-a}c(-\psi_3)}{\int_{-2a}^{-2s_1}c(t)e^{-t}dt}\\
			=&\lim_{s\rightarrow a-0}\frac{\int_{\{2s_1\le\psi_3<2a\}}|p_*(df_{z_0})|^2e^{b_1-a}c(-\psi_3)}{\int_{-2a}^{-2s_1}c(t)e^{-t}dt}.
	\end{split}
\end{equation}
For any $\epsilon>0$, as $\{G_{\Omega}(\cdot,z_0)=a\}$ is a compact set and $v_1=u_1$ on $\{G_{\Omega}(\cdot,z_0)\le a\}$, it follows from the upper Continuity of $v_1-u_1$ and Lemma \ref{l:extra} that there exists $a_2\in(a,a_1)$ such that 
\begin{equation}\label{eq:230111a}
	v_1-u_1<\epsilon
\end{equation}
on $\{a\le G_{\Omega}(\cdot,z_0)\le a_2\}$.
By equality \eqref{eq:0111b}, \eqref{eq:230111c} and inequality \eqref{eq:230111a}, 
\begin{equation}
	\nonumber
	\begin{split}
			&\lim_{s\rightarrow a+0}\frac{G(-2s;\psi_3,\tilde\varphi_3)-G(-2a;\psi_3,\tilde\varphi_3)}{\int_{-2s}^{+\infty}c(t)e^{-t}dt-\int_{-2a}^{+\infty}c(t)e^{-t}dt}\\
			=&\lim_{s\rightarrow a+0}\frac{\int_{\{2a\le\psi_3<2s\}}|F_1|^2e^{-2v_1-2\log|g_1|-a}c(-\psi_3)}{\int_{-2s}^{-2a}c(t)e^{-t}dt}\\
			=&\lim_{s\rightarrow a+0}\frac{\int_{\{2a\le\psi_3<2s\}}|p_*(df_{z_0})|^2e^{2u_1-2v_1+b_1-a}c(-\psi_3)}{\int_{-2s}^{-2a}c(t)e^{-t}dt}\\
			\ge&e^{-2\epsilon}\lim_{s\rightarrow a+0}\frac{\int_{\{2a\le\psi_3<2s\}}|p_*(df_{z_0})|^2e^{b_1-a}c(-\psi_3)}{\int_{-2s}^{-2a}c(t)e^{-t}dt}\\
			=&e^{-2\epsilon}\lim_{s_1\rightarrow a-0}\frac{G(-2s_1;\psi_3,\varphi_3)-G(-2a;\psi_3,\varphi_3)}{\int_{-2s_1}^{+\infty}c(t)e^{-t}dt-\int_{-2a}^{+\infty}c(t)e^{-t}dt},
	\end{split}
\end{equation}
which implies 
\begin{equation}
\nonumber\lim_{s\rightarrow a+0}\frac{G(-2s;\psi_3,\tilde\varphi_3)-G(-2a;\psi_3,\tilde\varphi_3)}{\int_{-2s}^{+\infty}c(t)e^{-t}dt-\int_{-2a}^{+\infty}c(t)e^{-t}dt}\ge \lim_{s_1\rightarrow a-0}\frac{G(-2s_1;\psi_3,\varphi_3)-G(-2a;\psi_3,\varphi_3)}{\int_{-2s_1}^{+\infty}c(t)e^{-t}dt-\int_{-2a}^{+\infty}c(t)e^{-t}dt}.
\end{equation}
As $G(h^{-1}(r);\psi_3,\tilde\varphi_3)$ is concave on $[0,\int_0^{+\infty}c(t)e^{-t}dt]$ and $G(t;\psi_3,\tilde\varphi_3)=G(t;\psi_3,\varphi_3)$ for any $t\ge-2a$, then we have
\begin{equation}
	\label{eq:01111a}
	\lim_{s\rightarrow a+0}\frac{G(-2s;\psi_3,\tilde\varphi_3)-G(-2a;\psi_3,\tilde\varphi_3)}{\int_{-2s}^{+\infty}c(t)e^{-t}dt-\int_{-2a}^{+\infty}c(t)e^{-t}dt}= \lim_{s_1\rightarrow a-0}\frac{G(-2s_1;\psi_3,\tilde\varphi_3)-G(-2a;\psi_3,\tilde\varphi_3)}{\int_{-2s_1}^{+\infty}c(t)e^{-t}dt-\int_{-2a}^{+\infty}c(t)e^{-t}dt}.
\end{equation}
As $G(h^{-1}(r);\psi_3,\tilde\varphi_3)$ is linear on $[0,\int_{-2a}^{+\infty}c(t)e^{-t}dt]$ and $[\int_{-2a}^{+\infty}c(t)e^{-t}dt,\int_0^{+\infty}c(t)e^{-t}dt]$, Equality \eqref{eq:01111a} deduces that $G(h^{-1}(r);\psi_3,\tilde\varphi_3)$ is linear on $[0,\int_{0}^{+\infty}c(t)e^{-t}dt]$. Then Theorem \ref{thm:e2} shows that the two statements in Theorem \ref{thm:char} hold.
\end{proof}

\section{Proofs of Theorem \ref{counterexample2}, Theorem \ref{counterexample} and Example \ref{e:counterexample3}}

\label{sec:p4}

In this section, we prove Theorem \ref{counterexample2}, Theorem \ref{counterexample} and Example \ref{e:counterexample3}.

\begin{proof}[Proof of Theorem \ref{counterexample2}]
	By Theorem \ref{thm:char} and Remark \ref{r:not linear}, $G(-\log r)$ is linear on $[0,e^{2a}]$ and $[e^{2a},+\infty)$, but  $G(-\log r)$ is not  linear on $[0,1]$.
	Following from Lemma \ref{l:notconvex}, we get that $-\log G(t)$	is not convex on $[0,+\infty)$.
\end{proof}

Now we prove Theorem \ref{counterexample}.
\begin{proof}[Proof of Theorem \ref{counterexample}]
	We prove Theorem \ref{counterexample} by contradiction: if not, then $-\log G_k(t)$ is convex on $[0,+\infty)$ for any $k\ge0$. Denote that 
	$$\tilde G_k(t)=G_k(2(k+1)t).$$
	By definition, 
	\begin{equation}
		\label{eq:1231c}\tilde G_k(t)\le \lambda(\{G_{\Omega}(\cdot,z_0)<-t\})
	\end{equation}
	holds for any $t\ge0$ and any $k\ge0$. Note that there exists $F_{k,t}\in\mathcal{O}(\{G_{\Omega}(\cdot,z_0)<-t\})$ such that $(F_{k,t}-1,z_0)\in\mathcal{I}(2(k+1)G_{\Omega}(\cdot,z_0))_{z_0}$ and 
	\begin{equation}
		\label{eq:1231b}\tilde G_k(t)=\int_{\{G_{\Omega}(\cdot,z_0)<-t\}}|F_{k,t}|^2.
	\end{equation}
	As $\tilde G_k(t)$ is increasing with respect to $k$ and $\lim_{t\rightarrow+\infty}\tilde G_k(t)\le \lambda(\{G_{\Omega}(\cdot,z_0)<-t\})$, we know that there exists a subsequence of $\{F_{k,t}\}_{k\ge0}$ denoted by $\{F_{k_l,t}\}_{l\ge0}$, which uniformly converges to a holomorphic function $F$ on $\{G_{\Omega}(\cdot,z_0)<-t\}$ on any compact subset of $\{G_{\Omega}(\cdot,z_0)<-t\}$. Since $(F_{k_l,t}-1,z_0)\in\mathcal{I}(2(k_l+1)G_{\Omega}(\cdot,z_0))_{z_0}$ and $\{G_{\Omega}(\cdot,z_0)<-t\}$ is connected, we know that $F\equiv1$. By Fatou's Lemma, inequality \eqref{eq:1231c} and equality \eqref{eq:1231b}, we have
	\begin{displaymath}
		\begin{split}
			\lambda(\{G_{\Omega}(\cdot,z_0)<-t\})&=\int_{\{G_{\Omega}(\cdot,z_0)<-t\}}1\\
			&=\int_{\{G_{\Omega}(\cdot,z_0)<-t\}}\lim_{l\rightarrow+\infty}|F_{k_l,t}|^2\\
			&\le\liminf_{l\rightarrow+\infty}\int_{\{G_{\Omega}(\cdot,z_0)<-t\}}|F_{k_l,t}|^2\\
			&=\liminf_{l\rightarrow+\infty}\tilde G_{k_l}(t)\\
			&\le \lambda(\{G_{\Omega}(\cdot,z_0)<-t\}),
		\end{split}
	\end{displaymath}
	which shows $\lambda(\{G_{\Omega}(\cdot,z_0)<-t\})=\lim_{t\rightarrow+\infty}\tilde G_k(t)$. As $-\log G_k(t)$ is convex on $[0,+\infty)$, we get that $-\log \lambda(\{G_{\Omega}(\cdot,z_0)<-t\})$ is convex on $[0,+\infty)$.
	
	By the concavity of $G_0(-\log r)$, 
	\begin{displaymath}
		\begin{split}
			G_0(2t)\ge e^{-2t}G_0(0),
		\end{split}
	\end{displaymath}  
	which implies that 
	\begin{equation}
		\label{eq:1231d}
		\begin{split}
			-2t-\log\lambda(\{G_{\Omega}(\cdot,z_0)<-t\})\le -2t-\log G_0(2t)\le-\log G_0(0)<+\infty.
		\end{split}
	\end{equation}
	As $-2t-\log\lambda(\{G_{\Omega}(\cdot,z_0)<-t\})$ is a convex function on $[0,+\infty)$,  inequality \eqref{eq:1231d} shows that $-2t-\log\lambda(\{G_{\Omega}(\cdot,z_0)<-t\})$ is decreasing on $[0,+\infty)$. Hence, $s(t):=e^{2t}\lambda(\{G_{\Omega}(\cdot,z_0)<-t\})$ is increasing on $[0,+\infty)$. Combining Lemma \ref{l:BZ}, we get that $s(t)$ is a constant function. Lemma \ref{l:BZ2} shows that $\Omega$ is a disc, which contradicts to the assumption in Theorem \ref{counterexample}. Thus, there exists $k\ge0$ such that $-\log G_k(t)$ is not convex on $[0,+\infty)$.
\end{proof}

Finally, we prove Example \ref{e:counterexample3}.
\begin{proof}[Proof of Example \ref{e:counterexample3}]
	Let $F$ be any holomorphic function on $\{\psi<-t\}$ satisfying $F^{(j)}(o)=j!a_j$, then $F=\sum_{0\le j\le k}a_jz^j+\sum_{j>k}b_jz^j$ on $\{\psi<-t\}$ (Taylor expansion). Thus, 
	\begin{equation}
		\nonumber
		\begin{split}
			\int_{\{\psi<-t\}}|F|^2&=\int_{\{2(k+1)\log|z|<-t\}}\bigg|\sum_{0\le j\le k}a_jz^j+\sum_{j>k}b_jz^j\bigg|^2\\
			&\ge \int_{\{2(k+1)\log|z|<-t\}}\bigg|\sum_{0\le j\le k}a_jz^j\bigg|^2\\
			&=\sum_{0\le j\le k}\frac{a_j\pi}{j+1}e^{-\frac{j+1}{k+1}t},
		\end{split}
	\end{equation}
	which shows that $G(t)=\sum_{0\le j\le k}\frac{a_j\pi}{j+1}e^{-\frac{j+1}{k+1}t}$ for any $t\ge0$.
	
	Denote that $h(t):=-\log G(t)$ on $[0,+\infty)$. Then we have
	\begin{displaymath}
		h''(t)=\frac{\left(\sum_{0\le j\le k}c_jd_je^{-d_jt}\right)^2-\left(\sum_{0\le j\le k}c_jd_j^2e^{-d_jt}\right)\left(\sum_{0\le j\le k}c_je^{-d_jt}\right)}{\left(\sum_{0\le j\le k}c_je^{-d_jt}\right)^2},
	\end{displaymath}
	where $c_j=\frac{a_j\pi}{j+1}$ and $d_j=\frac{j+1}{k+1}$.
	By Cauchy-Schwarz inequality, $h''\le0$. 
	Note that exist $j_1$ and $j_2$ satisfying $j_1\not=j_2$, $a_{j_1}\not=0$ and $a_{j_2}\not=0$. Since $\frac{c_jd_j^2e^{-d_jt}}{c_je^{-d_jt}}=d_j^2$ and $d_{j_1}\not=d_{j_2}$, then $h''(t)<0$ for any $t\ge0$. Thus, $-\log G(t)$ is concave on $[0,+\infty)$. 
\end{proof}

%%%------------------------------------------------------------------------

\vspace{.1in} {\em Acknowledgements}. The authors would like to thank Dr. Zhitong Mi for checking the manuscript. The second author was supported by National Key R\&D Program of China 2021YFA1003100 and NSFC-11825101. The third named author was supported by China Postdoctoral Science Foundation BX20230402 and 2023M743719.


\begin{thebibliography}{100}


\bibitem{BGY-concavity5}S.J. Bao, Q.A. Guan and Z. Yuan, Concavity property of minimal $L^2$ integrals with Lebesgue measurable gain \uppercase\expandafter{\romannumeral5} -- fibrations over open Riemann surfaces,
J. Geom. Anal. 33 (2023), no. 6, Paper No. 179, 73 pp.

\bibitem{BGY-concavity6}S.J. Bao, Q.A. Guan and Z. Yuan, Concavity property of minimal $L^2$ integrals with Lebesgue measurable gain \uppercase\expandafter{\romannumeral6}: fibrations over products of open Riemann surfaces, arXiv:2211.05255. 

\bibitem{BGMY-concavity7}S.J. Bao, Q.A. Guan, Z.T. Mi and Z. Yuan, Concavity property of minimal $L^2$ integrals with Lebesgue measurable gain \uppercase\expandafter{\romannumeral7} -- negligible weights, In: Hirachi, K., Ohsawa, T., Takayama, S., Kamimoto, J. (eds) The Bergman Kernel and Related Topics. HSSCV 2022. Springer Proceedings in Mathematics $\&$ Statistics, vol 447. Springer, Singapore. https://doi.org/10.1007/978-981-99-9506-6\_1.

\bibitem{Blogsub}
B. Berndtsson, Subharmonicity properties of the {B}ergman kernel and some other
functions associated to pseudoconvex domains. Ann. Inst. Fourier (Grenoble), 56(6), 1633-1662 (2006).

\bibitem{BL16}B. Berndtsson and L. Lempert, A proof of the Ohsawa-Takegoshi theorem with sharp estimates, J. Math. Soc. Japan, 68 (2016), 1461-1472.

\bibitem{BZ15}Z. Blocki and W. Zwonek,
Estimates for the Bergman kernel and the multidimensional Suita conjecture, New York J. Math. 21 (2015), 151-161.




\bibitem{DemaillySoc}J.-P. Demailly,
Multiplier ideal sheaves and analytic methods in algebraic geometry, School on Vanishing Theorems and Effective Result in Algebraic Geometry (Trieste,2000),1-148, ICTP lECT. Notes, 6, Abdus Salam Int. Cent. Theoret. Phys., Trieste, 2001.


\bibitem{DemaillyAG}J.-P. Demailly,
Analytic Methods in Algebraic Geometry,
Higher Education Press, Beijing, 2010.


\bibitem{DEL}J.-P. Demailly, L. Ein and R. Lazarsfeld,
A subadditivity property of multiplier ideals,
Michigan Math. J. 48 (2000) 137-156.


\bibitem{DK01}J.-P. Demailly and J. Koll\'ar,
Semi-continuity of complex singularity exponents and K\"ahler-Einstein metrics on Fano orbifolds,
Ann. Sci. \'Ec. Norm. Sup\'er. (4) 34 (4) (2001) 525-556.

\bibitem{DP03}J.-P. Demailly and T. Peternell,
A Kawamata-Viehweg vanishing theorem on compact K\"ahler manifolds,
J. Differential Geom. 63 (2) (2003) 231-277.


\bibitem{federer}H. Federer,
Curvature measures,
Trans. Amer. Math. Soc. 93 (1959), 418-491.

\bibitem{OF81}O. Forster, Lectures on Riemann surfaces, Graduate Texts in Mathematics, Vol. 81, Springer-Verlag, New York-Berlin, 1981.
	
\bibitem{G-R}
 H. Grauert and R. Remmert,
  Coherent analytic sheaves, Grundlehren der mathematischen
	Wissenchaften, 265, Springer-Verlag, Berlin, 1984.
	
	
	
\bibitem{Guenancia}H. Guenancia,
Toric plurisubharmonic functions and analytic adjoint ideal sheaves, Math. Z. 271 (3-4) (2012) 1011-1035.
	
	
\bibitem{G2018}Q.A. Guan,
General concavity of minimal $L^2$ integrals related to multiplier
sheaves,
arXiv:1811.03261.v4.
	
\bibitem{Guan2019}Q.A. Guan, A proof of Saitoh's conjecture for conjugate Hardy $H^2$ kernels. J. Math. Soc. Japan 71 (2019), no. 4, 1173-1179.
	

\bibitem{guan_sharp}Q.A. Guan, A sharp effectiveness result of Demailly's strong openness conjecture,
Adv. in Math. 348 (2019) 51-80.


\bibitem{guan-20}Q.A. Guan, Decreasing equisingular approximations with analytic singularities, J. Geom. Anal. 30 (2020), no. 1, 484-492.

\bibitem{GM_Sci} Q.A. Guan and Z.T. Mi, Concavity of minimal $L^2$ integrals related to multiplier ideal
sheaves on weakly pseudoconvex K\"ahler manifolds, Sci. China Math. 65 (2022), no. 5, 887-932.


\bibitem{GM}
Q.A. Guan and Z.T. Mi, Concavity of minimal $L^2$ integrals related to multiplier ideal
sheaves, Peking Math. J. 6 (2023), no. 2, 393-457.

\bibitem{GMY}
Q.A. Guan, Z.T. Mi and Z. Yuan,
	Concavity property of minimal $L^2$ integrals with Lebesgue measurable gain \uppercase\expandafter{\romannumeral2}, https://www.researchgate.net/publication/354464147, see also arXiv:2211.00473.
	
\bibitem{GMY-boundary2}Q.A. Guan, Z.T. Mi and Z. Yuan, Boundary points, minimal $L^2$ integrals and concavity property  \uppercase\expandafter{\romannumeral2}: weakly pseudoconvex K\"ahler manifolds, Sci. China Math. (2024). DOI:10.1007/s11425-022-2257-3
	
		
\bibitem{GMY-boundary3}Q.A. Guan, Z.T. Mi and Z. Yuan, Boundary points, minimal $L^2$ integrals and concavity property  \uppercase\expandafter{\romannumeral3} -- linearity on Riemann surfaces, arXiv:2203.15188.


\bibitem{GY-lp-effe}Q.A. Guan and Z. Yuan, Effectiveness of strong openness property in $L^p$, Proc. Amer. Math. Soc. 151 (2023), no. 10, 4331-4339.

\bibitem{GY-twisted}Q.A. Guan and Z. Yuan, Twisted version of strong openness property in $L^p$, arXiv:2109.00353.

\bibitem{GY-support}Q.A. Guan and Z. Yuan,  An optimal support function related to the strong openness conjecture, J. Math. Soc. Japan 74 (2022), no. 4, 1269-1293.

	
\bibitem{GY-concavity}Q.A. Guan and Z. Yuan, Concavity property of minimal $L^2$ integrals with Lebesgue measurable gain, Nagoya Math. J. 252 (2023), 842-905.


\bibitem{GY-concavity3}Q.A. Guan and Z. Yuan, Concavity property of minimal $L^2$ integrals with Lebesgue measurable gain \uppercase\expandafter{\romannumeral3}: open Riemann surfaces, https://www.researchgate.net/publication/356171464, see also 	arXiv:2211.04951.

\bibitem{GY-concavity4}Q.A. Guan and Z. Yuan, Concavity property of minimal $L^2$ integrals with Lebesgue measurable gain \uppercase\expandafter{\romannumeral4}: product of open Riemann surfaces, Peking Math. J. 7 (2024), no. 1, 91-154.


\bibitem{GY-saitoh}Q.A. Guan and Z. Yuan, A weighted version of Saitoh's conjecture, accepted by Publ. Res. Inst. Math. Sci., arXiv:2207.10976.


\bibitem{GZSOC}Q.A. Guan and X.Y. Zhou,
A proof of Demailly's strong openness conjecture, Ann. of Math.
(2) 182 (2015), no. 2, 605-616.

\bibitem{GZeff}Q.A. Guan and X.Y. Zhou,
Effectiveness of Demailly's strong openness conjecture and
related problems, Invent. Math. 202 (2015), no. 2, 635-676.











\bibitem{Lazarsfeld}R. Lazarsfeld,
Positivity in Algebraic Geometry. \uppercase\expandafter{\romannumeral1}. Classical Setting: Line Bundles and Linear Series. Ergebnisse der Mathematik und ihrer Grenzgebiete. 3. Folge. A Series of Modern Surveys in Mathematics [Results in Mathematics and Related Areas. 3rd Series. A Series of Modern Surveys in Mathematics], 48. Springer-Verlag, Berlin, 2004;\\
R. Lazarsfeld,
Positivity in Algebraic Geometry. \uppercase\expandafter{\romannumeral2}. Positivity for vector bundles, and multiplier ideals. Ergebnisse der Mathematik und ihrer Grenzgebiete. 3. Folge. A Series of Modern Surveys in Mathematics [Results in Mathematics and Related Areas. 3rd Series. A Series of Modern Surveys in Mathematics], 49. Springer-Verlag, Berlin, 2004.






\bibitem{Nadel}A. Nadel,
Multiplier ideal sheaves and K\"ahler-Einstein metrics of positive scalar curvature, Ann. of Math. (2) 132 (3) (1990) 549-596.


\bibitem{S-O69}L. Sario and K. Oikawa, Capacity functions, Grundl. Math. Wissen. 149, Springer-Verlag, New York, 1969. 



\bibitem{Siu96}
Y.T. Siu,
The Fujita conjecture and the extension theorem of Ohsawa-Takegoshi,
 Geometric Complex Analysis, World Scientific, Hayama, 1996, pp. 223-277.
 
 
 
 \bibitem{Siu05}Y.T. Siu,
Multiplier ideal sheaves in complex and algebraic geometry,
Sci. China Ser. A 48 (suppl.) (2005) 1-31.






\bibitem{Siu09}Y.T. Siu,
Dynamic multiplier ideal sheaves and the construction of rational curves in Fano manifolds,
Complex Analysis and Digital Geometry, in: Acta Univ. Upsaliensis Skr. Uppsala Univ. C Organ. Hist., vol.86, Uppsala Universitet, Uppsala, 2009, pp. 323-360.


\bibitem{Tian}G. Tian,
On K\"ahler-Einstein metrics on certain K\"ahler manifolds with $C_1(M)>0$, Invent. Math. 89 (2) (1987) 225-246.




\bibitem{Tsuji}M. Tsuji, Potential theory in modern function theory, Maruzen Co., Ltd., Tokyo, 1959. 


\bibitem{xuzhou}W. Xu and X.Y. Zhou, Optimal $L^2$ extensions of openness type, arXiv:2202.04791v2.


\end{thebibliography}
\end{document}